\documentclass[12pt]{article}

\usepackage{amsmath,amsfonts,amsthm,amssymb,graphicx}
\usepackage[mathscr]{euscript}
\usepackage{tikz,pgfplots}

\usepackage[a4paper,left=2cm,right=2cm,top=3cm,bottom=3cm]{geometry}

\newcommand{\bnabla}{\boldsymbol{\nabla}}

\newcommand{\btau}{\boldsymbol{\tau}}

\newcommand{\bW}{\boldsymbol{W}}
\newcommand{\bWh}{\boldsymbol{W}\!_h}
\newcommand{\bw}{\boldsymbol{w}}

\newcommand{\CP}{C_\mathrm{P}}

\newcommand{\cT}{\mathcal{T}}

\newcommand{\ddiv}{\operatorname{div}}

\newcommand{\dx}[1][\bfx]{\,\mathrm{d}#1}

\newcommand{\Hdiv}[1][\Omega]{\b{H}(\ddiv,#1)}
\newcommand{\Ieff}{I_{\text{eff}}}

\newcommand{\norm}[1]{\left\|#1\right\|}

\newcommand{\osc}{\operatorname{osc}}
\newcommand{\R}{\mathbb{R}}
\newcommand{\rmd}{\mathrm{d}}

\newcommand{\tkappa}{\widetilde\kappa}
\newcommand{\trinorm}[1]{|\!|\!|#1|\!|\!|}

\newtheorem{theorem}{Theorem}[section]
\newtheorem{definition}[theorem]{Definition}

\usepackage{algorithm,algpseudocode}
\usepackage[hidelinks]{hyperref}

\numberwithin{equation}{section}

\pgfplotsset{%
  every axis/.append style={%
    y label style={at={(0.1,1.0)},anchor=south west,rotate=-90,color=black},
    yminorgrids
  }
}
\pgfplotsset{compat=1.9}

\date{April 3, 2020}

\begin{document}

\title{%
Guaranteed a posteriori error bounds for low rank tensor approximate solutions
}

\author{Sergey Dolgov\thanks{S. Dolgov is thankful to the Engineering and Physical Sciences Research Council (UK) for the support through Fellowship EP/M019004/1. {\bf Email:}~{\tt s.dolgov@bath.ac.uk}}
\\
University of Bath, Department of Mathematical Sciences\\
Claverton Down, BA2 7AY Bath, United Kingdom
\\
\and 
Tom\'a\v s Vejchodsk\'y\thanks{T. Vejchodsk\'y gratefully acknowledges the support of the Czech Science Foundation, project no.~20-01074S, and the institutional support RVO 67985840. {\bf Email:}~{\tt vejchod@math.cas.cz}}
\\
  Institute of Mathematics,
  Czech Academy of Sciences,
\\  
  {\v Z}itn{\'a} 25, CZ-115 67 Prague 1,
  Czech Republic
}  

\maketitle

\begin{abstract}
We propose a guaranteed and fully computable upper bound on the energy norm of the error in low-rank Tensor Train (TT) approximate solutions of (possibly) high dimensional reaction-diffusion problems.
The error bound is obtained from Euler--Lagrange equations for a complementary flux reconstruction problem, which are solved in the low-rank TT representation using the block Alternating Linear Scheme.
This bound is guaranteed to be above the energy norm of the total error, including the discretization error, the tensor approximation error, and the error in the solver of linear algebraic equations,
although quadrature errors, in general, can pollute its evaluation.
Numerical examples with the Poisson equation and the Schr\"odinger equation with the Henon-Heiles potential in up to 40 dimensions are presented to illustrate the efficiency of this approach.
\end{abstract}



\section{Introduction}
This paper deals with the linear second-order elliptic partial differential equation of a reaction-diffusion type:
\begin{equation}
\label{eq:modpro}
-\Delta u + \kappa^2 u = f \quad \text{in }\Omega,\qquad
u = 0 \quad \text{on }\partial\Omega,
\end{equation}
where $\Omega \subset \R^d$ is a $d$-dimensional hyperrectangle, i.e. a Cartesian product of $d$ intervals.
The reaction coefficient $\kappa = \kappa(x)$ is assumed variable and can attain zero values in general.
The homogeneous Dirichlet boundary conditions are considered for simplicity.
Besides classical reaction-diffusion models in 2--3 dimensions,
implicit time stepping schemes for the Schr\"odinger equation \cite{meyer-mctdh-book-2009,meyer-henon-2002}
require to solve an equation of the form \eqref{eq:modpro},
in arbitrarily high dimension.

This problem is solved by the finite element method using low-rank tensor approximations, see Section \ref{se:tensor}, the book \cite{hackbusch-2012} or surveys \cite{larskres-survey-2013,bokh-surv-2015} for details.
This tensor approach enables us to compute an approximate solution even in high dimensional cases.
Since the number of degrees of freedom grows exponentially with the dimension $d$, the traditional approaches are prohibitively expensive for higher values of $d$.
This phenomenon is known as the curse of dimensionality \cite{bellman-dyn-program-1957} and the low-rank tensor approximation methods allow to break it in many practical cases,
reducing the computational costs and memory demands from exponential to polynomial in $d$.
The main idea is to consider the expansion coefficients of the finite element solution
in the form
of a $d$-dimensional tensor and approximate it by low-rank tensor decomposition.
In particular, we use the simple and robust Tensor Train (TT) decomposition \cite{osel-tt-2011}.

Low-rank tensor approximations bring further error to the computed solution. If a tensor is compressed from the full representation using the singular value decomposition \cite{osel-tt-2011}, this error can be controlled.
However, in practice, the full storage is not possible, and one computes a TT representation directly, using iterative interpolation or solution techniques.
In this case, it is challenging to obtain guaranteed and sharp estimates of the approximation error. Some recent results \cite{bachmayr-sparse-or-lr-2017} rely on a particular solution method that might be not the fastest one.
Guaranteed a priori estimates on the convergence rate \cite{ds-amen-2014,ushmaev-tt-2013} and approximation error \cite{tee-tensor-2003,uschmajew-approx-rate-2013,griebel-svd-sg-2013} are often too pessimistic.

In this paper, we propose a guaranteed a posteriori error estimator for a low-rank high dimensional solution, which is independent of a particular approximation algorithm, while being locally efficient up to higher-order terms \cite{AinVej:2018}.

This estimator is based on a complementary problem \cite{Complement:2010}.
It is a second-order elliptic partial differential equation that can be naturally discretized by Raviart--Thomas finite elements.
We approximate its solution by a low-rank tensor and use it to compute the sharp upper bound on the energy norm of the total error of the solution of the original reaction-diffusion problem.

This type of guaranteed a posteriori error bounds for linear second-order elliptic partial differential equations is already studied for many decades in the low dimensional case. The idea can be traced back to the method of hypercircle \cite{PraSyn:1947}.
After decades of development it attracted a lot of attention in recent years, see books \cite{AinOde:2000,LadPel2005,MoiMau2017,NeiRep:2004,Rep:2008}, papers \cite{AinBab:1999,CaiZha:2010,HanSteVoh:2012,LucWoh:2004,ParDie:2017,ParSanDie2009}, and references there in.
Guaranteed and robust error bounds for the particular reaction-diffusion problem \eqref{eq:modpro} were derived in
\cite{CheFucPriVoh:2009} for the vertex-centred finite volume method and in~\cite{AinBab:1999,robustaee:2010,AinVej:2014,AinVej:2018,SmeVoh:2018} for the finite element method.
A posteriori error estimates applicable to separable solutions were explored in~\cite{khoromskij-hom-2017}.

The rest of the paper is organized as follows.
Section~\ref{se:aee} derives the guaranteed upper bound on the total error of the finite element solution.
Section~\ref{se:fluxrec} introduces the complementary problem for the flux reconstruction.
Section~\ref{se:tensor} presents the main idea of low-rank tensor approximations, including the Cartesian grid, indexing, tensor train decomposition, the alternating scheme for solving systems of linear algebraic equations, tensorization of the complementary problem and the block alternating linear scheme for its solution, tensorized Gauss--Legendre quadrature, and a specific procedure for evaluation of the error estimator.
Section~\ref{se:numex} shows numerical results for the Poisson problem, reaction-diffusion problem, and Schr\"odinger equation with Henon--Heiles potential.
Finally, Section~\ref{se:concl} draws the conclusions and ideas for further research.

\section{Guaranteed a posteriori error bound}
\label{se:aee}

The well known Sobolev space $H^1_0(\Omega)$ consists of square integrable functions with square integrable distributional derivatives and with zero traces on the boundary $\partial\Omega$.
Denoting by $(\cdot,\cdot)$ both the $L^2(\Omega)$ and $[L^2(\Omega)]^d$ inner products,
the weak solution of problem \eqref{eq:modpro} is introduced as a function $u \in H^1_0(\Omega)$ such that
\begin{equation}
\label{eq:weakf}
(\bnabla u , \bnabla v) + (\kappa^2 u, v) = (f, v)
\quad\forall v \in H^1_0(\Omega).
\end{equation}
To guarantee integrability we consider $\kappa \in L^\infty(\Omega)$ and $f \in L^2(\Omega)$.

This problem is solved numerically using the standard finite element method of the first order, and the finite element solution is further approximated by a low-rank tensor as described below in Section~\ref{se:tensor}.
However, at this point, the particular details about the numerical solution are not important, because the guaranteed error bound, we will introduce, is independent of the used numerical method and applies to arbitrary conforming approximation $u_h \in H^1_0(\Omega)$ of $u$.

In order to introduce the guaranteed error bound, we denote by
$\cT_h$ the usual finite element mesh of the domain $\Omega$. More precisely,
$\cT_h$ is a set of closed $d$-dimensional hyperrectangles, called elements, which
form a face-to-face partition of $\Omega$.
Symbol $\Pi_K : L^2(K) \rightarrow \mathbb{Q}_{1,1,\dots,1}(K)$ stands for the $L^2(K)$-orthogonal
projector to the space $\mathbb{Q}_{1,1,\dots,1}(K)$ consisting of functions defined in the element $K\in\cT_h$ and being linear in each of their $d$ variables.
It is useful to introduce symbol $\Pi$ for elementwise concatenation of projectors $\Pi_K$ such that $(\Pi f)|_K = \Pi_K f$ for all $K\in\cT_h$.
Further, we denote by $\norm{\cdot}$ the norm in $L^2(\Omega)$ and by $\norm{\cdot}_K$ the norm in $L^2(K)$ for $K\in\cT_h$,
using the same symbol for both scalar and vector quantities.
The energy norm is given as $\trinorm{v}^2 = \norm{\bnabla v}^2 + \norm{\kappa v}^2$ for all $v \in H^1_0(\Omega)$
and its restriction to elements $K\in\cT_h$ as $\trinorm{v}_K^2 = \norm{\bnabla v}_K^2 + \norm{\kappa v}_K^2$.
It is convenient to recall the Poincar\'e inequality
\begin{equation}
  \label{eq:Poincare}
  \norm{ v } \leq \CP \norm{ \bnabla v } \quad \forall v \in H^1_0(\Omega).
\end{equation}
In the particular setting of the hyperrectangle and homogeneous Dirichlet boundary conditions the optimal value of $\CP$ is known to be
\begin{equation}
  \label{eq:CP}
  \CP = \pi^{-1} \left( \sum_{i=1}^d L_i^{-2} \right)^{-1/2},
\end{equation}
where $L_i$, $i=1,2,\dots,d$, are lengths of sides of the hyperrectangle $\Omega$,
see \cite{Mikhlin:1986}. For Neumann or some instances of mixed Dirichlet--Neumann boundary conditions, similar analytical formulas can be derived.
For general domains $\Omega$ and general types of boundary conditions, guaranteed upper bound on the optimal value of the Poincar\'e constant $\CP$ can be computed numerically, because $\CP = \lambda_1^{-1/2}$, where $\lambda_1$ is the first eigenvalue of the Laplacian in $\Omega$ with the corresponding boundary conditions and methods for guaranteed lower bounds on $\lambda_1$ exist, see e.g. recent works \cite{CarGal2014,CarGed2014,Liu2015,GoeHau1985,CanDusMadStaVoh2018,SebVej:2014} and references therein.
However, the case of general, e.g. polytopic, domains discretized by general, e.g. simplicial, meshes is not considered here, because the tensor approximation techniques would not be applicable.

Since zero values of the reaction coefficient $\kappa$ cause technical difficulties, we introduce a constant shift parameter $\kappa_0 > 0$ and define
\begin{equation}
\label{eq:tkappa}
\tkappa(x) = \kappa(x) + \kappa_0.
\end{equation}
Note that in principle $\kappa_0$ could be considered piecewise constant respecting the Cartesian structure of the problem. However, for the sake of simpler exposition, we do not consider this special case.
In order to simplify the notation, we set $r = f - \kappa^2 u_h$ and denote by $h_K$ the diameter of $K$.

The \emph{error estimator}, \emph{local error indicators} and the \emph{oscillation} term are defined by rules
\begin{align}
\label{eq:eta}
  \eta^2(\btau) &= \sum_{K\in\cT_h} (\eta_K(\btau) + \osc_K(r))^2,
\\
\label{eq:etaK}
  \eta_K^2(\btau) &= \norm{\btau - \bnabla u_h}_K^2
      + \norm{\tkappa^{-1} ( \Pi_K r + \ddiv \btau) }_K^2,
\\
\label{eq:oscK}
  \osc_K(r) &= \min \left\{ h_K \pi^{-1} \norm{r - \Pi_K r}_K,  \norm{\kappa^{-1} (r - \Pi_K r)}_K \right\},
\end{align}
respectively.
Quantity $\btau \in \Hdiv$,
where $\Hdiv$ stands for the well-known space of square integrable vector fields with square integrable divergence,
is called the \emph{flux}. At this moment it can be arbitrary, but its specific choice will be discussed in Section~\ref{se:fluxrec} below.
Note that if the norm $\norm{\kappa^{-1} (r - \Pi_K r)}_K$ is not defined, typically if $\kappa = 0$ in $K$, then we consider this norm to be infinite and $\osc_K(r) = h_K \pi^{-1} \norm{r - \Pi_K r}_K$.
Further note that estimator \eqref{eq:eta} differs from \cite[eq. (5)]{Complement:2010} in the oscillation term \eqref{eq:oscK}. Therefore, the estimator \eqref{eq:eta} behaves robustly even for small values of $\kappa$.

The following theorem is a generalization of \cite[Theorem 1]{AinVej:2018} for the variable coefficient $\kappa$. It provides the guaranteed and fully computable upper bound on the energy norm of the total error.
\begin{theorem}
\label{th:main}
Let $u \in H^1_0(\Omega)$ be the weak solution given by \eqref{eq:weakf}.
Let $u_h \in H^1_0(\Omega)$, $\btau \in\Hdiv$, and $\kappa_0 > 0$ be arbitrary.
Then
\begin{equation}
\label{eq:upperbound}
\trinorm{u - u_h} \leq
  \eta(\btau) + \kappa_0\CP \left(
    \sum_{K\in\cT_h} \norm{\tkappa^{-1} ( \Pi_K r + \ddiv \btau) }_K^2
  \right)^{1/2}
\end{equation}
with $\CP$ given by \eqref{eq:CP}.
Moreover, if $0 < \operatorname{ess\,inf}_\Omega \kappa$ then estimate \eqref{eq:upperbound} holds with $\kappa_0 = 0$.
\end{theorem}
\begin{proof}
The notation $v = u - u_h$, $r=f - \kappa^2 u_h$, the weak formulation \eqref{eq:weakf}, and the divergence theorem
for $\btau \in\Hdiv$
yield identity
\begin{multline}
 \label{eq:erreq}
 \trinorm{u-u_h}^2 = (\bnabla u - \bnabla u_h, \bnabla v) + (\kappa^2 u - \kappa^2 u_h,v)
 + (\btau,\nabla v) + (\ddiv\btau,v)
 =
\\
 \sum_{K\in\cT_h} \left[ (\btau - \bnabla u_h, \bnabla v)_K + (\Pi_K r + \ddiv\btau, v)_K + (r - \Pi_K r, v)_K
 \right],
\end{multline}
where $(\cdot,\cdot)_K$ stand for the $L^2(K)$ inner product.
The last inner product on the right-hand side can be estimated by the Cauchy--Schwarz inequality in two different ways:
\begin{align*}
  (r - \Pi_K r, v)_K &\leq \norm{\kappa^{-1}(r - \Pi_K r)}_K \norm{\kappa v}_K \leq \norm{\kappa^{-1}(r - \Pi_K r)}_K \trinorm{v}_K,
\\
  (r - \Pi_K r, v)_K &= (r - \Pi_K r, v - \overline{v}_K)_K \leq \norm{r - \Pi_K r}_K \norm{v - \overline{v}_K}_K
\\
    &\leq h_K \pi^{-1} \norm{r - \Pi_K r}_K \trinorm{v}_K,
\end{align*}
where $\overline{v}_K = |K|^{-1} (v,1)_K$ stands for the integral average of $v$ over $K$ and
the Poincar\'e inequality $\norm{v - \overline{v}_K}_K \leq h_K \pi^{-1} \norm{\bnabla v}_K$ is used
\cite{PayWei:1960,Bebendorf:2003}.
Thus,
\begin{equation}
  \label{eq:oscKest}
  (r - \Pi_K r, v)_K \leq \osc_K(r) \trinorm{v}_K,
\end{equation}
where we recall that if $\norm{\kappa^{-1}(r - \Pi_K r)}_K$ is not defined due to vanishing $\kappa$ then it is considered to be infinity.

Since $\tkappa > 0$, we multiply and divide the second inner product on the right-hand side of \eqref{eq:erreq} by $\tkappa$, use the Cauchy--Schwarz inequality and \eqref{eq:oscKest} to obtain
\begin{multline}
\label{eq:errest}
 \trinorm{u-u_h}^2
 \leq
   \sum_{K\in\cT_h}
    \left( \norm{\btau - \bnabla u_h}_K \norm{\bnabla v}_K
        + \norm{\tkappa^{-1} (\Pi_K r + \ddiv\btau)}_K \norm{\tkappa v}_K
        + \osc_K(r) \trinorm{v}_K
    \right).
\end{multline}
The triangle inequality $\norm{\tkappa v}_K \leq \norm{\kappa v}_K + \norm{\kappa_0 v}_K$ and
the bound
\begin{multline*}
\norm{\btau - \bnabla u_h}_K \norm{\bnabla v}_K
        + \norm{\tkappa^{-1} (\Pi_K r + \ddiv\btau)}_K \norm{\kappa v}_K
\\
\leq \left( \norm{\btau - \bnabla u_h}_K^2 + \norm{\tkappa^{-1} (\Pi_K r + \ddiv\btau)}_K^2 \right)^{1/2}
     \left( \norm{\bnabla v}_K^2 + \norm{\kappa v}_K^2 \right)^{1/2}
\end{multline*}
provide an estimate
\begin{multline}
\label{eq:errest2}
 \trinorm{u-u_h}^2
 \leq
    \sum_{K\in\cT_h}
      \left( \eta_K(\btau) + \osc_K(r) \right) \trinorm{v}_K
  + \sum_{K\in\cT_h}
      \norm{\tkappa^{-1} (\Pi_K r + \ddiv\btau)}_K \norm{\kappa_0 v}_K.
\end{multline}
Separate application of the Cauchy--Schwarz inequality to the first and second sum yields
$$
\trinorm{u-u_h}^2
\leq
\eta(\btau) \trinorm{v}
 + \kappa_0 \norm{v} \left( \sum_{K\in\cT_h}
      \norm{\tkappa^{-1} (\Pi_K r + \ddiv\btau)}_K^2 \right)^{1/2}.
$$
Poincar\'e inequality \eqref{eq:Poincare} and notation $v = u - u_h$ now implies the inequality \eqref{eq:upperbound}.

Moreover, if $0 < \operatorname{ess\,inf}_\Omega \kappa$ then it is possible to choose $\kappa_0 = 0$.
Function $\tkappa$ is then positive and the whole proof, specifically \eqref{eq:errest}, remains valid.
\end{proof}

Theorem~\ref{th:main} can be generalized to the case of nonhomogeneous Neumann and mixed Dirichlet--Neumann boundary conditions in a similar way as in \cite{AinVej:2018}. However, Neumann boundary conditions bring forward additional terms in the definition of the error estimator and one more oscillation term. We wish to avoid these terms to simplify the exposition and present the main idea of the tensor approximation without unnecessary heavy notation.

Note that the upper bound \eqref{eq:upperbound} can be simplified for the price of its slight increase.
Indeed, using $\norm{\tkappa^{-1} (\Pi_K r + \ddiv\btau)}_K \leq \eta_K(\btau) \leq \eta_K(\btau) + \osc_K(r)$ in \eqref{eq:errest2}, we easily obtain bound
\begin{equation*}
\trinorm{u - u_h} \leq
  \left(1+\kappa_0^2\CP^2\right)^{1/2}\eta(\btau).
\end{equation*}
Further note that triangle inequality implies
\begin{equation}
\label{eq:tildeeta}
\eta(\btau) \leq \tilde\eta(\btau)
  = \left( \sum_{K\in\cT_h} \eta_K^2 \right)^{1/2}
  + \left( \sum_{K\in\cT_h} \osc_K^2(r) \right)^{1/2},
\end{equation}
and definition \eqref{eq:etaK} gives
\begin{equation}
\label{eq:sumetaKsq}
\sum_{K\in\cT_h} \eta_K^2
  = \norm{\btau - \bnabla u_h}^2 + \norm{\tkappa^{-1} ( \Pi r + \ddiv \btau) }^2.
\end{equation}
Since values $\kappa^{-1}$ in \eqref{eq:oscK} may cause quadrature errors and technical problems in higher dimension, we found the following estimate to be useful:
\begin{equation}
\label{eq:sumoscKsq}
\osc_K^2(r)
  \leq \min \left\{ h_K \pi^{-1}, \max_K \kappa^{-1} \right\} \norm{r - \Pi r}_K.
\end{equation}

Concerning the shift parameter $\kappa_0$, it should be chosen small. Ideally so small that $\kappa_0\CP$ is negligible with respect to 1 and $1 + \kappa_0 \CP \approx 1$.
If the reaction coefficient $\kappa$ is bounded away from zero in $\Omega$ then neither the shift $\kappa_0$ nor the Poincar\'e constant $\CP$ are needed and estimate \eqref{eq:upperbound} holds with $\kappa_0 = 0$.

An alternative upper bound on the energy norm of the error is based on the choice of
the flux $\btau$ such that it satisfies an equilibration condition, for example $\Pi_K r + \ddiv \btau = 0$ or $\int_K (\Pi_K r + \ddiv \btau) \dx = 0$ for all $K\in\cT_h$, see \cite{SmeVoh:2018}.
This alternative upper bound is similar to the bound presented in Theorem~\ref{th:main} and
the equilibration condition enables us to avoid technical difficulties with vanishing $\kappa$ and neither the shift parameter $\kappa_0$ nor the Poincar\'e constant $\CP$ are needed. However, we do not prefer this approach, because the equilibration condition cannot be satisfied exactly in practical computations due to round-off and tensor truncation errors and consequently, the upper bound on the error cannot be guaranteed. Moreover, the practical implementation of the equilibration condition is technically more involved, especially for high dimensional problems.

\section{Flux reconstruction}
\label{se:fluxrec}

Theorem~\ref{th:main} provides a guaranteed error bound for arbitrary flux $\btau \in \Hdiv$. However, in order to obtain a sharp error bound, the flux $\btau$ has to be chosen carefully. A natural idea is to minimize the error estimator $\eta^2(\btau)$ over a finite dimensional subspace $\bWh \subset \Hdiv$.
We employ the standard finite element technology and choose the Raviart--Thomas finite element space
$$
  \bWh = \{ \btau_h \in \Hdiv : \btau_h|_K \in \mathbf{RT}_1(K) \quad \forall K \in \cT_h \},
$$
where
$$
  \mathbf{RT}_1(K) = [ \mathbb{Q}_{2,1,\dots,1}(K), \mathbb{Q}_{1,2,\dots,1}(K), \dots, \mathbb{Q}_{1,1,\dots,2}(K)]^T
$$
is the local Raviart--Thomas space
\cite{RavTho1977}, \cite[sec. 2.4.4]{SolSegDol2004}
on the $d$-dimensional hyperrectangle $K$
and $\mathbb{Q}_{p_1,p_2,\dots,p_d}(K)$ stands for the space of polynomials in $K$ having degree at most $p_s$ in variable $x_s$, $s=1,2,\dots,d$.
Note that the fact that piecewise polynomial vector fields in $\bWh$ are all in $\Hdiv$ is equivalent to the continuity of their normal components over all hyperfaces in the mesh $\cT_h$.

Since the functional $\eta^2(\btau)$ is not quadratic, its minimization leads to a nonlinear problem. Therefore, we leave out the oscillation terms
and define the \emph{flux reconstruction} $\btau_h \in \bWh$ as the minimizer of the quadratic functional
\eqref{eq:sumetaKsq} over the Raviart--Thomas space $\bWh$.
The Euler--Lagrange equations corresponding to this minimization problem read
\begin{equation}
\label{eq:dualprob}
  \left(\tkappa^{-2} \ddiv \btau_h, \ddiv \bw_h\right) + (\btau_h , \bw_h)
  =
  \left( \tkappa^{-2} [ \Pi(\kappa^2 u_h) - \tkappa^2 u_h - \Pi f], \ddiv \bw_h \right)
\end{equation}
for all $\bw_h \in \bWh$.
Here, again, $(\cdot,\cdot)$ denote the $[L^2(\Omega)]^d$ and $L^2(\Omega)$ scalar products.
The problem of finding $\btau_h \in \bWh$ satisfying these equations is called the \emph{complementary problem}.
Note that the right-hand side of \eqref{eq:dualprob} is adjusted by using the divergence theorem.
Interestingly, if $\kappa$ is piecewise constant then projections $\Pi$ can be ommited on the right-hand side of \eqref{eq:dualprob}. If in addition, $\kappa$ is bounded away from zero in $\Omega$, and $\kappa_0 = 0$, then problem \eqref{eq:dualprob} is independent of $u_h$.

Let us note that the local efficiency for the local error indicators $\eta_K(\btau_h)$ is still open and out of the scope of this paper. However, a local efficiency result for a similar case, namely simplicial elements, piecewise constant $\kappa$, and a flux reconstruction computed by minimizing local analogies of the quadratic functional \eqref{eq:sumetaKsq} on patches of elements, is proved in \cite{AinVej:2018}.

Since the space $\bWh$ is finite dimensional, problem \eqref{eq:dualprob} is equivalent to a system of linear algebraic equations.
This system is solved by using low-rank tensor approximations as described in the following section.

\section{Low-rank tensor approximations of $u_h$ and $\btau_h$}
\label{se:tensor}

This section describes low-rank tensor approximations of the finite element solution $u_h$ and later also of the reconstructed flux $\btau_h$. The standard finite element method suffers from the curse of dimensionality: both the memory requirements and the computational time grow exponentially with $d$. If the dimension $d$ is much larger than $3$, as in a typical Schr\"odinger equation, then they become prohibitively large.
In this case, we approximate expansion coefficients representing $u_h$ and $\btau_h$ by suitable low-rank tensor decompositions.

\subsection{Cartesian grid and indexing}
The domain $\Omega = (a_1,b_1) \times \cdots \times (a_d, b_d)$ is a Cartesian product of intervals $(a_k, b_k)$, $k=1,\ldots,d$.
Within this domain, we introduce a Cartesian product grid,
$$
z(i) = \left(z_1(i_1), \ldots, z_d(i_d)\right), \qquad a_k = z_k(0) < \cdots < z_k(i_k) < \cdots < z_k(n_k) = b_k,
$$
for $k=1,\ldots,d$, where
$i_k$ are individual indices of $n_k+1$ nodes in the $k$-th direction,
and
\begin{equation}\label{eq:multiind}
 i = i_d + i_{d-1}(n_d+1)  + \cdots + i_1(n_d+1) \cdots (n_2+1), \quad i_k=0,\ldots,n_k,
\end{equation}
is the \emph{global} index of the point $z(i)$ in the whole Cartesian product grid.
Note that $i=0,\ldots, N-1$, where $N=(n_1+1) \cdots (n_d+1)$ is indeed the cardinality of the $d$-dimensional grid.
Further, we define $h_k(i_k) = z_k(i_k) - z_k(i_k-1)$ for $i_k = 1,2,\dots,n_k$.

Introducing
the finite elements
\begin{equation}
\label{eq:element}
K(i_1,\ldots,i_d) = [z_1(i_1-1),z_1(i_1)]\times \cdots \times [z_d(i_d-1), z_d(i_d)], \quad i_k=1,\ldots,n_k,
\end{equation}
we define the usual piecewise linear finite element space
$$
V_h = \{ v_h \in H^1_0(\Omega) : v_h|_{K(i_1,\ldots,i_d)} \in \mathbb{Q}_{1,1,\dots,1}(K(i_1,\ldots,i_d)),\ i_k=1,2,\dots,n_k,\ k=1,2,\dots,d \}
$$
and the finite element solution $u_h \in V_h$ of problem \eqref{eq:modpro} by the identity
\begin{equation}\label{eq:FEM}
(\bnabla u_h , \bnabla v_h) + (\kappa^2 u_h, v_h) = (f, v_h)
\quad\forall v_h \in V_h.
\end{equation}

Defining the standard univariate piecewise linear and continuous hat functions
$$
\varphi_{i_k}^{(k)}(x_k) = \left\{\begin{array}{ll}
\displaystyle \frac{x_k - z_k(i_k-1)}{h_k(i_k)}, & \quad\mbox{if } z_k(i_k-1)\le x_k \le z_k(i_k), \\
\displaystyle \frac{z_k(i_k+1) - x_k}{h_k(i_k+1)}, & \quad\mbox{if }  z_k(i_k)\le x_k \le z_k(i_k+1), \\
0, & \mbox{otherwise},
\end{array}\right.
$$
for $i_k = 1,2,\dots,n_k-1$, $k=1,2,\dots,d$, we express $u_h$ as
\begin{equation}\label{eq:hatd}
 u_h(x) = \sum_{i_1,\ldots,i_d=1}^{n_1-1,\ldots,n_d-1} \hat u(i) \varphi^{(1)}_{i_1}(x_1) \cdots \varphi^{(d)}_{i_d}(x_d),
\end{equation}
where $x = (x_1, x_2, \dots, x_d)$ and $\hat u \in \mathbb{R}^{(n_1+1)\cdots(n_d+1)}$ is a vector containing expansion coefficients of $u_h$ in the basis of $V_h$.
Note that the boundary entries of $\hat u$ (i.e. those corresponding to $i_k=0$ and $i_k=n_k$ for any $k$) are not used in \eqref{eq:hatd}.
For consistency, we set them to zero.

The Cartesian framework simplifies the construction of the weak form \eqref{eq:FEM}.
For example, for $\kappa = 0$, plugging \eqref{eq:hatd} into the finite element formulation \eqref{eq:FEM}, we obtain a structured linear system
$\hat A \hat u = \hat b$, where
\begin{equation}\label{eq:Alap}
 \hat A = L_1 \otimes M_2 \otimes \cdots \otimes M_d + \cdots + M_1 \otimes \cdots \otimes M_{d-1} \otimes L_d.
\end{equation}
Moreover, assuming a separable factorization
\begin{align*}
  f(x) &= f^{(1)}(x_1) f^{(2)}(x_2) \cdots f^{(d)}(x_d),
\end{align*}
we obtain a separable right-hand side
\begin{equation}\label{eq:bsep}
  \hat b = b_1 \otimes b_2 \otimes \cdots \otimes b_d.
\end{equation}
Here,
$$
(L_k)_{i_k,j_k} = \left(\frac{\rmd\varphi^{(k)}_{i_k}}{\rmd x_k}, \frac{\rmd\varphi^{(k)}_{j_k}}{\rmd x_k}\right), \quad (M_k)_{i_k,j_k} = \left(\varphi^{(k)}_{i_k}, \varphi^{(k)}_{j_k}\right), \quad \mbox{and} \quad(b_k)_{i_k} = (f^{(k)},\varphi^{(k)}_{i_k})
$$
are elements of the ``one-dimensional'' stiffness and mass matrices, and the load vector, respectively.
These scalar products are understood in the $L^2(a_k,b_k)$ sense, and
$i_k,j_k=1,\ldots,n_k-1$, $k=1,2,\dots,d$.
Moreover, entries corresponding to the boundary conditions are set to $(b_k)_{i_k} = (L_k)_{i_k,j_k} = (M_k)_{i_k,j_k} = 0$
whenever $i_k,j_k=0$ or $n_k$ and $i_k \neq j_k$, and $(L_k)_{i_k,i_k} = (M_k)_{i_k,i_k} = 1$ when $i_k=0$ or $n_k$.
Symbol $\otimes$ stands for the Kronecker product, defined for any matrices $A=\left[a_{i,j}\right]$ and $B$ as follows,
$$
A \otimes B = \begin{bmatrix}a_{1,1} B & \cdots & a_{1,m} B \\ \vdots & & \vdots \\ a_{n,1} B & \cdots & a_{n,m} B\end{bmatrix}.
$$

Although the number of entries of the matrix $\hat A$ and of the load vector $\hat b$
grows exponentially with $d$, the number of entries in ``one-dimensional'' factors $L_k$, $M_k$, and $b_k$ grows linearly.
Therefore, the idea of how to break the curse of dimensionality is to never compute Kronecker products and work directly with ``one-dimensional'' factors instead.
The following subsection briefly describes a more general low-rank decomposition for the tensors of expansion coefficients $\hat u$, load $\hat b$, and for the stiffness matrix $\hat A$.

\subsection{TT decomposition}

We employ the Tensor Train (TT) \cite{osel-tt-2011} low-rank tensor decomposition format.
The vector $\hat u$ is approximated by a vector $\tilde u$ (with elements $\tilde u(i)$) that admits the TT decomposition
\begin{equation}\label{eq:tt}
 \tilde u(i) = \sum_{\alpha_0,\ldots,\alpha_d=1}^{r_0,\ldots,r_d} u^{(1)}_{\alpha_0,\alpha_1}(i_1)  u^{(2)}_{\alpha_1,\alpha_2}(i_2)  \cdots  u^{(d)}_{\alpha_d-1,\alpha_d}(i_d).
\end{equation}
\begin{definition}\label{def:ttblocks}
The factors on the right-hand side of \eqref{eq:tt} are called \emph{TT blocks}, and can be seen in the following three equivalent forms representing the same data.
\begin{enumerate}
 \item All values $u^{(k)}_{\alpha_{k-1},\alpha_k}(i_k)$ for $\alpha_{k-1}=1,\ldots,r_{k-1}$, $\alpha_k=1,\ldots,r_k$ and $i_k=0,\ldots,n_{k}$ can be collected into a three-dimensional \emph{tensor} $u^{(k)} \in \mathbb{R}^{r_{k-1} \times (n_k+1) \times r_k}$ (hence the name tensor decompositions).
 \item For fixed $\alpha_{k-1},\alpha_k$, the \emph{slice} $u^{(k)}_{\alpha_{k-1},\alpha_k} \in \mathbb{R}^{n_k+1}$ is a (column) vector.
 This allows us to rewrite the TT decomposition \eqref{eq:tt} using Kronecker products,
\begin{equation}\label{eq:ttkron}
 \tilde u = \sum_{\alpha_0,\ldots,\alpha_d=1}^{r_0,\ldots,r_d} u^{(1)}_{\alpha_0,\alpha_1} \otimes u^{(2)}_{\alpha_1,\alpha_2} \otimes \cdots \otimes u^{(d)}_{\alpha_d-1,\alpha_d}.
\end{equation}
Note that this expansion coincides with \eqref{eq:bsep} for $r_0=\cdots=r_d=1$.
\item All values of a TT block can be collected into a \emph{vector} $\bar u^{(k)}(\zeta_k) = u^{(k)}_{\alpha_{k-1},\alpha_k}(i_k)$, $\bar u^{(k)} \in \mathbb{R}^{r_{k-1} (n_k+1)r_k}$, where $\zeta_k = \alpha_{k-1} + i_k r_{k-1} + (\alpha_k-1) (n_k+1) r_{k-1}$ is the global index comprised of $\alpha_{k-1},\alpha_k$ and $i_k$. This will allow us to formulate a linear system for entries of a chosen TT block in Section~\ref{sec:als}.
\end{enumerate}
\end{definition}
The auxiliary summation ranges $r_0,\ldots,r_d$ are called \emph{TT ranks}.
For consistency with the left-hand side, the border ranks are constrained to $r_0=r_d=1$, but the intermediate TT ranks can be larger than one,
and depend on the desired approximation error $\hat u - \tilde u$.
Assuming that all grid sizes and intermediate TT ranks are the same, $n_1=\cdots=n_d=n$ and $r_d=\cdots=r_{d-1}=r$,
we conclude that this so-called \emph{rank-$r$} TT decomposition contains $\mathcal{O}(dnr^2)$ unknowns.
If the TT rank $r$ is moderate, this can be much smaller than the original $(n+1)^d$ unknowns.

Many closed-form functions admit TT approximations (or even exact decompositions) of their coefficient tensors with low TT ranks \cite{osel-constr-2013,dk-qtt-tucker-2013}.
For example, consider the function $u(x) = \sin(\pi x_1)\cdots\sin(\pi x_d) \in H_0^1((0,1)^d)$ that can be approximated on the above Cartesian grid by
a rank-1 TT decomposition \eqref{eq:tt} with $u^{(k)}(i_k) = \sin(\pi i_k/n_k)$, $k=1,\ldots,d$.
Smooth functions can often be approximated by truncated series, where each term is low-rank, which gives a bound on the total rank \cite{tee-tensor-2003,uschmajew-approx-rate-2013}.
For example, the solution of the Poisson equation (i.e. problem \eqref{eq:modpro} with $\kappa=0$)
with a low-rank right-hand side can be approximated
by a tensor in the TT format with a relative error $\varepsilon$
and TT ranks bounded by $\mathcal{O}(\log^2 \varepsilon)$ \cite{grasedyck-kron-2004}.

In general, any tensor can be approximated by a TT decomposition using
$\mathcal{O}(dnr^2)$ entries from the tensor via cross interpolation algorithms \cite{ot-ttcross-2010,sav-qott-2014},
allowing $r$ to be large enough.
Of course, in practice we are interested in problems where $r$ can be taken small, e.g. independent of (or mildly dependent on) the dimension $d$ and the grid size $n$.

Structured matrices can also be represented in a TT format.
For example, one can prove \cite{khkaz-lap-2012} that the stiffness matrix \eqref{eq:Alap} can be represented as
\begin{equation}\label{eq:ttm}
 \hat A = \sum_{\beta_0,\ldots,\beta_d=1}^{R_0,\ldots,R_d} A^{(1)}_{\beta_0,\beta_1} \otimes \cdots \otimes A^{(d)}_{\beta_{d-1},\beta_d}
\end{equation}
with
$R_0 = R_d = 1$,
$R_1=\cdots=R_{d-1}=2$, and the following TT blocks:
$$
A^{(1)} = \left\{\begin{matrix}L_1 & M_1\end{matrix}\right\}, \quad
A^{(k)} = \left\{\begin{matrix}M_k & 0 \\ L_k & M_k \end{matrix}\right\}, \quad
A^{(d)} = \left\{\begin{matrix}M_d \\ L_d \end{matrix}\right\},
$$
where $k=2,\ldots,d-1$, the rows of the curly bracket matrices correspond to the rank index $\beta_{k-1}$, and the columns correspond to $\beta_k$ (e.g. $A^{(k)}_{1,1}=A^{(k)}_{2,2}=M_k$, $A^{(k)}_{2,1}=L_k$, $A^{(k)}_{1,2}=0$).
This representation clearly reduces the storage costs compared to the direct summation of Kronecker product terms in \eqref{eq:Alap}.

The product $\hat A \tilde u$ can be explicitly expressed as another TT decomposition with TT ranks $r_0R_0,\ldots,r_dR_d$ \cite{osel-tt-2011},
without expanding the Kronecker products.
Moreover, the result can be \emph{re-approximated} up to quasi-optimal ranks for the given accuracy using the singular value decomposition (SVD) with the cost proportional to $d$.

\subsection{Alternating scheme for solving linear equations}\label{sec:als}
First attempts to solve systems of linear algebraic equations with decomposed tensors were based on traditional iterative methods such as Richardson, CG, and GMRES
with matrix-vector products and other operations implemented on TT blocks followed by the rank truncation \cite{tobkress-param-2011,KhSch-Galerkin-SPDE-2011,balgras-htgmres-2013,dc-ttgmres-2013}.
This approach applies to any low-rank tensor decomposition format.
However, in realistic problems the TT ranks of intermediate (e.g. Krylov) vectors grow rapidly,
exceeding the optimal ranks of the solution significantly unless a good preconditioner is used.

A more robust technique that is commonly adopted nowadays is the Alternating Linear Scheme (ALS) \cite{holtz-ALS-DMRG-2012} and its extensions
\cite{DoOs-dmrg-solve-2011,ds-amen-2014,kressner-evamen-2014}.
These methods project the equations onto bases constructed from the TT decomposition of the solution itself
and therefore avoid decompositions of auxiliary vectors with high TT ranks.
We emphasize that the TT form \eqref{eq:tt} of the solution is crucial for this technique.

Let us start with the standard ALS algorithm, suitable for the primal problem \eqref{eq:FEM}.
The TT decomposition \eqref{eq:tt} can be rewritten as a linear map from the elements of a $k$-th TT block to the elements of $\tilde u$.
Given \eqref{eq:tt}, let us define partial TT decompositions
\begin{align}\label{eq:partialTTdecl}
U^{(<k)}_{\alpha_{k-1}} &= \sum_{\alpha_{0},\ldots,\alpha_{k-2}=1}^{r_{0},\ldots,r_{k-2}} u^{(1)}_{\alpha_{0},\alpha_1} \otimes \cdots \otimes u^{(k-1)}_{\alpha_{k-2},\alpha_{k-1}} &\in\mathbb{R}^{(n_1+1)\cdots(n_{k-1}+1)} \\
\label{eq:partialTTdecr}
U^{(>k)}_{\alpha_{k}} &= \sum_{\alpha_{k+1},\ldots,\alpha_{d}=1}^{r_{k+1},\ldots,r_{d}} u^{(k+1)}_{\alpha_{k},\alpha_{k+1}} \otimes \cdots \otimes u^{(d)}_{\alpha_{d-1},\alpha_{d}}   &\in\mathbb{R}^{(n_{k+1}+1)\cdots(n_d+1)}
\end{align}
for any $1\le k \le d$ and fixed $\alpha_{k-1},\alpha_k$.
These vectors for different values of $\alpha_{k-1},\alpha_k$ can be collected into a $\prod_{\ell=1}^{k-1}(n_{\ell}+1) \times r_{k-1}$ matrix $U^{(<k)}$ and a $r_{k} \times \prod_{\ell={k+1}}^{d}(n_{\ell}+1)$ matrix $U^{(>k)}$.
Moreover, we set $U^{(<1)} = U^{(>d)} = 1$ for uniformity of notation in what follows.
Now we introduce the so-called \emph{frame} matrix by replacing $u^{(k)}$ in \eqref{eq:tt} by the identity matrix $I$ of size $n_k+1$:
\begin{equation}\label{eq:frame}
 U_{\neq k} = U^{(<k)} \otimes I \otimes \left(U^{(>k)}\right)^\top \in \mathbb{R}^{\prod_{\ell=1}^{d}(n_{\ell}+1) \times (r_{k-1}(n_k+1)r_k)}.
\end{equation}
Using the vector form notation $\bar u^{(k)}$ for the $k$-th TT block (see Def. \ref{def:ttblocks}.3),
it is then easy to prove \cite{holtz-ALS-DMRG-2012} that the TT decomposition \eqref{eq:tt} is equivalent to a linear map
$$
\tilde u = U_{\neq k} \bar u^{(k)} \quad \mbox{for each} \quad k=1,\ldots,d.
$$

If we plug this decomposition into the linear system $\hat A \hat u = \hat f$ (replacing $\hat u$ by $\tilde u$),
we obtain an overdetermined linear system for $\bar u^{(k)}$ (and hence for elements of $u^{(k)}$).
The simplest way to resolve it
is to project the equation onto the same frame matrix $U_{\neq k}$,
\begin{equation}\label{eq:localsys}
 \left(U_{\neq k}^\top \hat A U_{\neq k}\right) \bar u^{(k)} = \left(U_{\neq k}^\top \hat f\right).
\end{equation}
The ALS algorithm now iterates over $k=1,\ldots,d$ (hence the name alternating), solving the reduced system \eqref{eq:localsys} of size $r_{k-1} (n_k+1) r_k$ in each step, see the summary in Algorithm~\ref{alg:als}.
\begin{algorithm}[htb]
\caption{ALS algorithm for the primal (scalar) linear system}
\label{alg:als}
\begin{algorithmic}[1]
 \Require A matrix in the TT format \eqref{eq:ttm}, an initial guess $\tilde u$ in the TT format \eqref{eq:tt}, a right-hand side $\hat b$ in a TT format equivalent to \eqref{eq:tt}, a relative stopping tolerance $\delta_p>0$.
 \State Initialise $\tilde u_{prev}=0$.
 \While{$\|\tilde u - \tilde u_{prev}\|_2>\delta_p \|\tilde u\|_2$}
   \State Copy $\tilde u_{prev} = \tilde u$.
   \For{$k=1,2,\ldots,d,d-1,\ldots,1$}
     \State Assemble and solve \eqref{eq:localsys}, using \eqref{eq:frame} and \eqref{eq:partialTTdecl}--\eqref{eq:partialTTdecr}.
     \State Overwrite all elements $u^{(k)}_{\alpha_{k-1},\alpha_k}(i_k) = \bar u^{(k)}(\zeta_k)$ according to Def. \ref{def:ttblocks}.3.
   \EndFor
 \EndWhile
\end{algorithmic}
\end{algorithm}

If the matrix $\hat A$ is symmetric positive definite then the projected system \eqref{eq:localsys}
can be related to an optimization problem.
Consequently, the ALS algorithm can be related to the nonlinear block Gauss--Seidel method and the local convergence can be proved \cite{ushmaev-tt-2013}.

For the numerical efficiency one can notice that the reduced matrix and right-hand side in \eqref{eq:localsys}
can be constructed from the TT blocks of $\hat A, \hat u$ and $\hat f$ without ever expanding the Kronecker products.
In a sequential iteration over $k$ as described in Alg.~\ref{alg:als},
partial projections can be cached such that the cost of each ALS step becomes independent of $d$ \cite{holtz-ALS-DMRG-2012,DoOs-dmrg-solve-2011}.
The TT blocks can be \emph{enriched} with auxiliary vectors, such as approximate residuals \cite{ds-amen-2014},
which provides a mechanism for increasing TT ranks and their adaptation to the desired accuracy.

\subsection{TT decomposition for the complementary problem}

Since the complementary solution $\btau_h \in \bW_h$ is a vector field with $d$ components, the corresponding linear operator in \eqref{eq:dualprob} has a $d \times d$ block structure.
Therefore it is important to adopt a particular TT decomposition for expansion coefficients of $\btau_h$.
This specific structure will be useful for the
tailored iterative solver described in Subsection \ref{se:alstau}.
As a by-product, we obtain a low-rank algorithm which might be efficient for the solution of more general equations as well.

The general idea for the TT decomposition of $\btau_h$ is similar to the decomposition of $u_h$.
However, the construction of the Cartesian product Raviart--Thomas space $\mathbf{RT}_1$ is more complicated.
The $s$-th component $\tau_{h,s}(x)$ of $\btau_h(x)$, $s=1,\ldots,d$, is a piecewise polynomial function
that is piecewise quadratic and continuous in the $s$-th variable and piecewise linear and discontinuous in all the other variables.
Let $\hat\varphi^{(k)}_{j_k}(x_k)$, $j_k=1,2,\dots,2n_k$, denote piecewise linear and discontinuous functions defined as
\begin{align*}
 \hat\varphi^{(k)}_{2\ell_k-1}(x_k) &=
   \left\{ \begin{array}{rl}
     \dfrac{z_k(\ell_k)-x_k}{h_k(\ell_k)}, &\quad
     \mbox{for } x_k \in \left[z_k(\ell_k-1),~z_k(\ell_k)\right],\\
     0, &\quad\mbox{elsewhere,}
   \end{array}\right.
\\
 \hat\varphi^{(k)}_{2\ell_k}(x_k) &=
   \left\{ \begin{array}{rl}
     \dfrac{x_k-z_k(\ell_k-1)}{h_k(\ell_k)}, &\quad
     \mbox{for } x_k \in \left[z_k(\ell_k-1),~z_k(\ell_k)\right], \\
     0, &\quad\mbox{elsewhere,}
         \qquad\qquad\qquad\qquad \ell_k = 1,2,\dots,n_k.
   \end{array}\right.
\end{align*}
Similarly, let $\phi^{(k)}_{i_k}(x_k)$, $i_k = 0,1,\dots,2n_k$, $k=1,2,\dots,d$, be the usual piecewise quadratic and continuous ``one-dimensional'' basis functions.
Then components of the vector field $\btau_h = (\tau_{h,1},\ldots,\tau_{h,d}) \in \bW_h$ can be expanded in these basis functions as follows
\begin{multline}\label{eq:q2d}
  \tau_{h,s}(x) = \sum_{\substack{j_1,\ldots,j_{s-1},\\ j_{s+1},\ldots,j_d=1}}^{2n_1,\ldots,2n_d} \sum_{i_s=0}^{2n_s} \hat\tau_s(j)
  \cdot \hat\varphi^{(1)}_{j_1}(x_1) \cdots \hat\varphi^{(s-1)}_{j_{s-1}}(x_{s-1}) \cdot \phi^{(s)}_{i_s}(x_s) \cdot \hat\varphi^{(s+1)}_{j_{s+1}}(x_{s+1}) \cdots \hat\varphi^{(d)}_{j_d}(x_d),
\end{multline}
where the global index for the vector $\hat\tau_s \in \mathbb{R}^{(2n_1)\cdots (2n_s+1) \cdots (2n_d)}$ is defined similarly to \eqref{eq:multiind} as follows,
\begin{equation}\label{eq:multiind_tau}
j = j_d  + \cdots + i_s (2n_{s+1})\cdots (2n_d) + \cdots + (j_1-1) (2n_2)\cdots (2n_{s-1}) (2 n_s+1) (2n_{s+1})\cdots (2n_d).
\end{equation}
This ansatz enables us to express the complementary problem \eqref{eq:dualprob} in the following block structure:
\begin{equation}\label{eq:dualblock}
\begin{bmatrix}
 B_{1,1} & \cdots & B_{1,d} \\
 \vdots & \cdots & \vdots \\
 B_{d,1} & \cdots & B_{d,d}
\end{bmatrix}
\begin{bmatrix}
\hat\tau_1 \\ \vdots \\ \hat\tau_d
\end{bmatrix}
=
\begin{bmatrix}
  \hat g_1 \\ \vdots \\ \hat g_d
\end{bmatrix},
\end{equation}
where
\begin{align}
\label{eq:Bss}
(B_{s,s})_{j,j'} &=  (\tkappa^{-2} \partial_s \psi^{(s)}_j,\partial_s \psi^{(s)}_{j'}) + (\psi^{(s)}_j,\psi^{(s)}_{j'}), \quad s=1,\ldots,d,
\\ \nonumber
(B_{s,k})_{j,j'} &= (\tkappa^{-2} \partial_s \psi^{(s)}_j,\partial_k \psi^{(k)}_{j'}),  \qquad s \neq k,
\\ \nonumber
(\hat g_s)_j &= (\tkappa^{-2}[\Pi(\kappa^2 u_h) - \tkappa^2 u_h - \Pi f], \partial_s \psi^{(s)}_j),
\quad
\\ \nonumber
\psi^{(s)}_j(x) &= \hat\varphi^{(1)}_{j_1}(x_1)\cdots\hat\varphi^{(s-1)}_{j_{s-1}}(x_{s-1}) \phi^{(s)}_{i_s}(x_s)  \hat\varphi^{(s+1)}_{j_{s+1}}(x_{s+1})\cdots\hat\varphi^{(d)}_{j_{d}}(x_d),
\end{align}
$\partial_s$ stands for the partial derivative $\partial/\partial x_s$,
and global indices $j,j'=1,\ldots,(2n_1)\cdots (2n_s+1) \cdots (2n_d)$ are linked to $j_1,\ldots,j_{s-1},i_s,j_{s+1},\ldots,j_d$ in accordance with \eqref{eq:multiind_tau}.

Since all blocks $B_{s,k}$ are nonzero, the components of
the vector $\hat\tau = [\hat\tau_1^\top,\dots,\hat\tau_d^\top]^\top$
are coupled through \eqref{eq:dualblock}, and they should be approximated in the same TT decomposition,
in order to apply the ALS technique.
This requires all components to have same dimensions.
Therefore, we expand the ranges of the summation indices $j_1,\ldots,j_{s-1},j_{s+1},\ldots,j_d$ in \eqref{eq:q2d} to start from zero, expand the
vectors $\hat\tau_s$ by zeros, and the $\mathbf{RT}_1$ space by any dummy functions accordingly.
Under this convention, we can write the finite element solution as
\begin{multline}
\label{eq:tauhsx}
  \tau_{h,s}(x) = \sum_{i_1,\ldots,i_d=0}^{2n_1, \ldots 2n_d} \hat\tau_s(i) \cdot \hat\varphi^{(1)}_{i_1}(x_1) \cdots \hat\varphi^{(s-1)}_{i_{s-1}}(x_{s-1}) \cdot \phi^{(s)}_{i_s}(x_s) \cdot \hat\varphi^{(s+1)}_{i_{s+1}}(x_{s+1}) \cdots \hat\varphi^{(d)}_{i_d}(x_d),
\end{multline}
with the global index
\begin{equation}\label{eq:multiind_comp}
i = i_d + i_{d-1} (2n_d+1) + \cdots + i_1 (2n_2+1) \cdots (2n_d+1).
\end{equation}
Consequently, all blocks $B_{s,k}$, $s,k = 1,2,\dots,d$, are of size
$\prod_{\ell=1}^d (2 n_\ell + 1) \times \prod_{\ell=1}^d (2 n_\ell + 1)$.

\subsection{Block ALS for the complementary problem}
\label{se:alstau}

In principle, we could consider the component index $s$ as another independent variable, stack $\hat \tau_s$ vertically into a long vector $\hat\tau$, approximate this vector in a TT decomposition with $d+1$ blocks,
and apply the standard ALS, Algorithm~\ref{alg:als}, to the block linear system \eqref{eq:dualblock}.
However, due to the special meaning of the index $s$, this ALS algorithm
might be inefficient, as we demonstrate by the numerical experiments in Section~\ref{se:numex}.
Instead, we incorporate the component index $s$ into the TT block which is being evaluated in the current ALS step,
using the so-called block TT format~\cite{dkos-eigb-2014}.

The block TT approximation $\tilde\tau_s \approx \hat\tau_s$ is defined as follows,
\begin{equation}\label{eq:btt}
\tilde\tau_s = \sum_{\alpha_0,\ldots,\alpha_d=1}^{r_0,\ldots,r_d} \tau^{(1)}_{\alpha_0,\alpha_1} \otimes \cdots \otimes \tau^{(k-1)}_{\alpha_{k-2},\alpha_{k-1}} \otimes t^{(k)}_{\alpha_{k-1},\alpha_k,s} \otimes \tau^{(k+1)}_{\alpha_{k},\alpha_{k+1}} \otimes \cdots \otimes \tau^{(d)}_{\alpha_{d-1},\alpha_{d}},
\end{equation}
where $s$ can appear in any $k$-th TT block for $k=1,\ldots,d$,
which is denoted by $t^{(k)}$ to distinguish it from the other TT blocks $\tau^{(k)}$.
Again, for fixed $\alpha_{\ell}$ and $s$, all factors in the right-hand side are vectors, e.g. $\tau^{(\ell)}_{\alpha_{\ell-1},\alpha_\ell} \in \mathbb{R}^{2n_\ell+1}$ and $t^{(k)}_{\alpha_{k-1},\alpha_k,s} \in \mathbb{R}^{2n_k+1}$.
The benefit of this block TT decomposition for the ALS method stems from the form of the frame matrix.
Similarly to \eqref{eq:partialTTdecl}--\eqref{eq:partialTTdecr}, we introduce
\begin{align}\label{eq:blockPartL}
T^{(<k)}_{\alpha_{k-1}} &= \sum_{\alpha_{0},\ldots,\alpha_{k-2}=1}^{r_{0},\ldots,r_{k-2}} \tau^{(1)}_{\alpha_{0},\alpha_1} \otimes \cdots \otimes \tau^{(k-1)}_{\alpha_{k-2},\alpha_{k-1}} &\in\mathbb{R}^{(2n_1+1)\cdots(2n_{k-1}+1)} \\
\label{eq:blockPartR}
T^{(>k)}_{\alpha_{k}} &= \sum_{\alpha_{k+1},\ldots,\alpha_{d}=1}^{r_{k+1},\ldots,r_{d}} \tau^{(k+1)}_{\alpha_{k},\alpha_{k+1}} \otimes \cdots \otimes \tau^{(d)}_{\alpha_{d-1},\alpha_{d}}   &\in\mathbb{R}^{(2n_{k+1}+1)\cdots(2n_d+1)},
\end{align}
and notice that the $k$-th frame matrix
\begin{equation}\label{eq:frameblock}
T_{\neq k} = T^{(<k)} \otimes I \otimes \left(T^{(>k)}\right)^\top,
\end{equation}
cf. \eqref{eq:frame}, consists
of TT blocks without the component index $s$.
This allows us to reduce the block matrix in \eqref{eq:dualblock} component by component, preserving and exploiting its block structure
for efficient solution of the reduced problem.

Specifically, the \emph{block ALS} method \cite{bdos-sb-2016} computes the block TT approximation \eqref{eq:btt} of $\hat\tau$ by iterating over $k=1,\ldots,d$ and
solving in each step the reduced system
\begin{equation}\label{eq:locblock}
\begin{bmatrix}
 T_{\neq k}^\top B_{1,1} T_{\neq k} & \cdots & T_{\neq k}^\top B_{1,d}T_{\neq k} \\
 \vdots & \cdots & \vdots \\
 T_{\neq k}^\top B_{d,1}T_{\neq k} & \cdots & T_{\neq k}^\top B_{d,d} T_{\neq k}
\end{bmatrix}
\begin{bmatrix}
 \bar t^{(k)}_1 \\ \vdots \\ \bar t^{(k)}_d
\end{bmatrix}
=
\begin{bmatrix}
  T_{\neq k}^\top \hat g_1 \\ \vdots \\ T_{\neq k}^\top \hat g_d
\end{bmatrix}
\end{equation}
of size $r_{k-1}(2n_k+1)r_k d$.
Here we introduce a vector notation similar to Def.~\ref{def:ttblocks}.3,
$$
\bar t^{(k)}_s(\zeta_k) = t^{(k)}_{\alpha_{k-1},\alpha_k,s}(i_k), \quad \mbox{for}\quad \zeta_k = \alpha_{k-1} + i_k r_{k-1} + (\alpha_k-1)r_{k-1}(2 n_k+1).
$$
Note that when we switch to the next step ($k-1$ or $k+1$), the index $s$ in the block TT format \eqref{eq:btt} needs to be moved to the corresponding ($k-1$ or $k+1$) TT block,
such that the new frame matrix remains independent of $s$.
This operation can be performed using the singular value decomposition (SVD).

For example, suppose we need to move the index $s$ from the $k$-th TT block to the $(k+1)$-th TT block.
First, we stretch $t^{(k)}$ into a $r_{k-1}(2n_k+1) \times d r_k$ matrix by grouping indices $\alpha_{k-1}$ and $i_k$ into a new row index, and
$s$ and $\alpha_k$ into a new column index,
\begin{equation}\label{eq:tk-reshape}
T^{(k)}(\alpha_{k-1},i_k;~s, \alpha_k) = t^{(k)}_{\alpha_{k-1},\alpha_k,s}(i_k).
\end{equation}
Computing SVD of $T^{(k)}$ and truncating the singular values below the desired error threshold, we obtain
\begin{equation}\label{eq:tk-svd}
T^{(k)}(\alpha_{k-1},i_k;~s, \alpha_k) \approx \sum_{\alpha_k'=1}^{r_k'} P(\alpha_{k-1},i_k;~\alpha_k') \Sigma(\alpha_k',\alpha_k') Q(\alpha_k';~s,\alpha_k),
\end{equation}
where $P$ and $Q$ are matrices of left and right singular vectors, respectively, and $\Sigma(\alpha_k',\alpha_k')$
are singular values.
In the new block TT decomposition of $\tilde\tau_s$ the TT block $t^{(k)}$ in \eqref{eq:btt} is replaced by
\begin{equation}\label{eq:new-block-tauk}
\tau^{(k)}_{\alpha_{k-1},\alpha_k'}(i_k) = P(\alpha_{k-1},i_k;~\alpha_k')
\end{equation}
and the TT block $\tau^{(k+1)}$ is replaced by
\begin{equation}\label{eq:new-block-tk}
t^{(k+1)}_{\alpha_k',\alpha_{k+1},s}(i_{k+1}) = \sum_{\alpha_k=1}^{r_k} \Sigma(\alpha_k',\alpha_k') Q(\alpha_k';~s,\alpha_k) \tau^{(k+1)}_{\alpha_k,\alpha_{k+1}}(i_{k+1}).
\end{equation}
Notice that the new block TT decomposition has the same form as \eqref{eq:btt} except that $s$ is located in the $(k+1)$-th block, and the $k$-th TT rank is $r_k'$.
This new rank may differ from $r_k$ in general, but in practical computations the difference is often insignificant.
This procedure can be continued or reversed in order to place $s$ into any desired block, as described in Algorithm~\ref{alg:als-block}.

\begin{algorithm}[htb]
\caption{Block ALS algorithm for the complementary linear system}
\label{alg:als-block}
\begin{algorithmic}[1]
 \Require Matrices $B_{s,k}$ in the TT formats \eqref{eq:ttm}, initial guesses $\tilde \tau_s$ in the TT format \eqref{eq:btt}, right-hand sides $\hat g_s$ in a TT format equivalent to \eqref{eq:tt}, a relative approximation tolerance $\delta_c>0$.
 \State Initialise $\tilde \tau^{prev}=0$.
 \While{$\|\tilde \tau - \tilde \tau^{prev}\|_2>\delta_c \|\tilde \tau\|_2$}
   \State Copy $\tilde \tau^{prev} = \tilde \tau$.
   \For{$k=1,2,\ldots,d,d-1,\ldots,1$}
     \State Assemble and solve \eqref{eq:locblock}, using \eqref{eq:frameblock} and \eqref{eq:blockPartL}--\eqref{eq:blockPartR}.
     \State Compute SVD as in \eqref{eq:tk-reshape}--\eqref{eq:tk-svd}, choose $r_k'$ such that $\|T^{(k)}-P\Sigma Q\|_F \le \delta_c \|T^{(k)}\|_F$.
     \If {$k$ increases in the next step}
       \State Overwrite $\tau^{(k)}$ and $t^{(k+1)}$ as shown in \eqref{eq:new-block-tauk}--\eqref{eq:new-block-tk}.
     \Else
       \State Overwrite $\tau^{(k)}$ and $t^{(k-1)}$ similarly, copying $Q$ into $\tau^{(k)}$ and $P\Sigma$ into $t^{(k-1)}$.
     \EndIf
   \EndFor
 \EndWhile
\end{algorithmic}
\end{algorithm}

The size of \eqref{eq:locblock} might still be rather large for the direct solution.
Typical values of TT ranks range from $10$ to $50$, and the grid sizes $n_k$ can range from tens to hundreds.
However, the matrix inherits the TT decomposition, and hence a fast matrix-vector product can be implemented for iterative solvers \cite{DoOs-dmrg-solve-2011}.
Specifically, we use GMRES with a block Jacobi preconditioner with respect to the rank indices $\alpha_{k-1},\alpha_k$ in the solution.
Recall that $\bar t^{(k)}_s$ is enumerated by three independent indices $\alpha_{k-1},\alpha_k,i_k$,
and hence each matrix $\hat B_{s,l} = T_{\neq k}^\top B_{s,l}T_{\neq k}$ can in turn be seen as a three-level block matrix,
$\hat B_{s,l} = [\hat B_{s,l}(\alpha_{k-1},\alpha_k,i_k; ~\alpha_{k-1}',\alpha_k',i_k')]$.
The preconditioner is constructed by extracting the diagonal with respect to the first two levels, i.e.
$\tilde B_{s,l} = [\hat B_{s,l}(\alpha_{k-1},\alpha_k,i_k; ~\alpha_{k-1},\alpha_k,i_k') \delta(\alpha_{k-1},\alpha_{k-1}') \delta(\alpha_{k},\alpha_{k}')].$

In order to construct \eqref{eq:locblock} efficiently, we need to decompose the components of the matrix $B_{s,l}$ and the right-hand side $\hat g_s$ in the TT formats \eqref{eq:ttm} and \eqref{eq:tt}.
This step relies crucially on the approximation of $\tkappa^{-2}(x)$, $f$, and other functions.
We collocate them on a special quadrature grid using the TT-Cross method \cite{ot-ttcross-2010}.
This method together with the Cartesian structure of the finite elements enables us to easily construct
the linear and bilinear forms in TT formats for both the primal and the complementary problem.
The TT-Cross method is always run with the same relative approximation threshold $\delta_c$ as the one used in Algorithm~\ref{alg:als-block}.

\subsection{Tensorized Gauss--Legendre quadrature}\label{sec:quad}
Error indicators \eqref{eq:etaK}, oscillation terms \eqref{eq:oscK}, as well as entries of involved matrices and vectors are given as integrals over finite elements. These integrals are computed by tensorized Gauss--Legendre quadrature.
Utilizing the Cartesian structure of the finite elements, we use one-dimensional Gauss--Legendre quadrature
in each interval $[z_k(i_k-1), z_k(i_k)]$, constituting the element \eqref{eq:element}.
Assuming that $m$ quadrature nodes are introduced in each interval,
we end up with the total of $m n_k$ quadrature nodes $a_k \le y_k(1) < \cdots < y_k(mn_k) \le b_k$ for the $k$-th direction,
and with the corresponding quadrature weights $w_k(\ell_k)$, $\ell_k=1,\ldots,mn_k$.

Now, all integrated functions in weak formulations \eqref{eq:weakf}, \eqref{eq:dualprob} and definitions \eqref{eq:etaK}, \eqref{eq:oscK} are approximated by vectors
of collocation values at $y(\ell) = (y_1(\ell_1),\ldots,y_d(\ell_d))$, where the global index is introduced as previously,
\begin{equation}\label{eq:multiind_quad}
\ell = \ell_d + (\ell_{d-1}-1) (mn_d) + \cdots (\ell_1-1) (mn_2)\cdots (mn_d).
\end{equation}
For example, to evaluate the first scalar product in \eqref{eq:Bss}, we introduce a vector $\hat\sigma \in \mathbb{R}^{(mn_1)\cdots (mn_d)}$ with elements
$$
\hat \sigma(\ell) = \tkappa^{-2}(y_1(\ell_1),\ldots,y_d(\ell_d)),
$$
being the values of $\tkappa^{-2}$ at the quadrature nodes,
and approximate it by the TT-Cross algorithm as
$$
\hat \sigma \approx \tilde\sigma = \sum_{\alpha_0,\ldots,\alpha_d=1}^{r_0,\ldots,r_d} \sigma^{(1)}_{\alpha_0,\alpha_1} \otimes \cdots \otimes \sigma^{(d)}_{\alpha_{d-1},\alpha_d}.
$$
This enables us to efficiently compute
the matrix elements. For example, for $s \neq k$,
\begin{equation}\label{eq:W3}
(B_{s,k})_{j,j'} = (\tkappa^{-2} \partial_s \psi^{(s)}_j, \partial_k \psi^{(k)}_{j'}) \approx \sum_{\ell_1,\ldots,\ell_d=1}^{mn_1,\ldots,mn_d} w_1(\ell_1) \cdots w_d(\ell_d) \tilde\sigma(\ell) \partial_s \psi^{(s)}_j(y(\ell)) \partial_k \psi^{(k)}_{j'} (y(\ell)),
\end{equation}
where $\ell$ is the total index according to \eqref{eq:multiind_quad}, and $j,j'$ are total indices according to \eqref{eq:multiind_comp}.
Due to the product structure of $\tilde\sigma$ and $\psi^{(s)}$,
all terms in \eqref{eq:W3} corresponding to a fixed direction can be grouped together, and the sum over $\ell$ can be implemented as a product of
sums over individual $\ell_k$ with the total computational cost being linear in $d$.
Moreover, the loops over $j$ and $j'$ can be factorized as well, such that the whole matrix on the left-hand side of \eqref{eq:W3} can be written in the matrix TT format \eqref{eq:ttm} block by block, e.g. $\sum_{\ell_q} w_q(\ell_q) \sigma^{(q)}_{\alpha_{q-1},\alpha_q}(\ell_q) \hat\varphi^{(q)}(y_q(\ell_q)) \hat\varphi^{(q)} (y_q(\ell_s))$ becomes the $q$-th TT block ($q\neq s \neq k$), and so on.
Since the elements have finite support, the total cost of this operation is $\mathcal{O}(d n r^2)$,
where $r$ is the maximal TT rank of $\tilde\sigma$.

\subsection{Evaluation of the error estimator}
\label{se:evalee}

Importantly, Theorem~\ref{th:main} holds for arbitrary $u_h \in H^1_0(\Omega)$ and arbitrary $\btau_h \in \Hdiv$.
Therefore, the upper bound property \eqref{eq:upperbound} holds even if $u_h$ and $\btau_h$ are polluted by various errors,
including iteration errors in the ALS algorithm, SVD truncation errors,
and quadrature errors in matrix elements.
However, the upper bound \eqref{eq:upperbound} needs to be evaluated exactly.
Given the approximate solution $u_h$ and the complementary solution $\btau_h$ in the TT format,
it might be difficult to evaluate $\eta(\btau_h)$ in \eqref{eq:upperbound} without truncation errors.
Specifically, if we are interested in the \emph{elementwise} error estimation, we run the TT-Cross to approximate the terms in~\eqref{eq:etaK},
then we approximate the pointwise square root of $\eta_K^2(\btau)$, and finally we use the TT-Cross again to compute $\eta(\btau)$ from~\eqref{eq:eta}.
If we choose the TT-Cross approximation tolerance much smaller than that used for the ALS solvers then this approximation of~\eqref{eq:upperbound} will likely remain an upper bound on the error. However, it is not guaranteed and, moreover,
it might result in large TT ranks of some of intermediate quantities.

On the other hand,
the tensor approximation errors can be avoided if
$\eta(\btau)$ is replaced by
a slightly larger quantity
$\tilde\eta(\btau)$ given by \eqref{eq:tildeeta}--\eqref{eq:sumetaKsq},
because $\tilde\eta(\btau)$
has a
favourable separable structure.
After solving primal and complementary problems, we are given TT approximations of
$\btau_h$, $\mathrm{div}~\btau_h$, $\bnabla u_h$, $\tkappa^{-1}$ and $\Pi r$.
We interpolate them at quadrature nodes $y_k(\ell_k)$
and multiply the error terms by square roots of the quadrature weights, e.g.
\begin{equation}
\begin{split}
\tilde\eta_1(\ell) & = \sqrt{w_1(\ell_1)\cdots w_d(\ell_d)} \left(\btau_h(y(\ell)) - \bnabla u_h(y(\ell)) \right), \\
\tilde\eta_2(\ell) & = \sqrt{w_1(\ell_1)\cdots w_d(\ell_d)} \tkappa^{-1}(y(\ell)) \left(\Pi r(y(\ell)) + \mathrm{div}~\btau_h(y(\ell)) \right),
\end{split}
\label{eq:tildeeta_quad}
\end{equation}
$\ell_k = 1,2,\dots,m n_k$, $k=1,2,\dots,d$, and $\ell$ is according to \eqref{eq:multiind_quad}.
These vectors depend \emph{polylinearly} on $\btau_h,u_h,\tilde\kappa^{-1}$ and $r$, and thus their TT decompositions can be constructed \emph{exactly},
followed by taking exact 2-norms \cite{osel-tt-2011} of the TT formats,
$$
\sum_{K \in \mathcal{T}_h} \eta_K^2(\btau_h) = \|\tilde\eta_1\|_2^2 + \|\tilde\eta_2\|_2^2,
$$
where we recall that $\eta_K^2(\btau_h)$ satisfy~\eqref{eq:sumetaKsq}.
An upper bound on the oscillation term can be easily computed by using~\eqref{eq:sumoscKsq}, where $\|r - \Pi r\|$ and $\min\left\{h_K \pi^{-1}, \max_K \kappa^{-1}\right\}$ are computed separately.

Hence, the computed estimator $\tilde\eta(\btau)$ is polluted by the round-off and quadrature errors due to $\kappa$ and $f$, only. In special cases, for example for constant $\kappa$ and polynomial $f$, the quadrature errors vanish and the computed bound is guaranteed up to the round-off errors.

\section{Numerical examples}
\label{se:numex}

\subsection{Constant reaction coefficient}
As an example, we consider the reaction-diffusion problem \eqref{eq:modpro} in the unit cube $\Omega = (0,1)^d$
with the constant reaction coefficient $\kappa^2$, and the right-hand side
$$
f(x) = 8\sum_{k=1}^{d} \prod_{\substack{i=1\\i\neq k}}^{d} \left(1-4(x_i-0.5)^2\right) + \kappa^2 \prod_{k=1}^{d} \left(1-4(x_i-0.5)^2\right).
$$
chosen such that the exact solution is
$$
u(x) = \prod_{k=1}^{d} \left(1-4(x_k-0.5)^2\right).
$$
We are to investigate how the accuracy of the proposed error estimator depends on the number of grid points
$n = n_1 = \cdots = n_d$,
the dimension $d$, and the reaction coefficient $\kappa^2$.

For the computation of matrix elements and error estimators as described in Section \ref{sec:quad},
we use $m=4$ Gauss--Legendre points in each interval of the grid.
Since $\kappa$ is constant and $f$ is a polynomial of degree at most $2$ in each variable, this quadrature rule is exact.
The value of $m=4$ is used in order to define the same rule throughout all experiments in the paper.

For comparison, we present error estimators
$\eta(\btau_h)$ and $\tilde\eta(\btau_h)$ given by \eqref{eq:eta} and \eqref{eq:tildeeta}
as well as the corresponding guaranteed error bounds given by the right-hand side of \eqref{eq:upperbound}:
\begin{align*}
E(\btau_h) &=   \eta(\btau_h) + \kappa_0\CP \left(
    \sum_{K\in\cT_h} \norm{\tkappa^{-1} ( \Pi_K r + \ddiv \btau_h) }_K^2
  \right)^{1/2},
\\
\tilde E(\btau_h) &= \tilde\eta(\btau_h) + \kappa_0\CP \left(
    \sum_{K\in\cT_h} \norm{\tkappa^{-1} ( \Pi_K r + \ddiv \btau_h) }_K^2
  \right)^{1/2}.
\end{align*}
Estimator $\tilde\eta(\btau_h)$ and error bound $\tilde E(\btau_h)$ are evaluated as described in Section~\ref{se:evalee}.
Estimator $\eta(\btau_h)$ and error bound $E(\btau_h)$ cannot be computed without tensor truncation errors and their approximate values are computed using the TT Cross algorithm.
The accuracy of these error estimators and error bounds is measured by indices of effectivity
$
\Ieff^{(\mathcal{E})}=\mathcal{E} / \trinorm{u-u_h},
$
where $\mathcal{E}$ stands for $\eta$, $E$, $\tilde\eta$, and $\tilde E$, respectively.

We start with the dependence of indices of effectivity on the number of grid points $n$. We fix $\kappa=0$ and $d=3$.
The shift parameter is chosen as $\kappa_0 = 0.1$ and TT approximation thresholds
for the primal and complementary problem
as $\delta_p = 10^{-3}$ and $\delta_c = 10^{-7}$, respectively, see Algorithms~\ref{alg:als} and \ref{alg:als-block}.
Figure~\ref{fig:lap:n} (left) shows differences of $\Ieff$ from $1$.
We can notice that all estimators are almost indistinguishable for this example and converge with the second order $\mathcal{O}(h^2)$ as the grid is refined. This is due to sufficiently small $\kappa_0$.
The only and very minor difference caused by the second term in \eqref{eq:upperbound} can be spotted only for very fine grids.
Error bounds $E$ and $\tilde E$ are guaranteed by Theorem~\ref{th:main} to have indices of effectivity above one. Interestingly, all indices of effectivity of estimators $\eta$ and $\tilde\eta$ are above one as well.
Further notice that the optimal estimator $\eta(\btau_h)$ is indistinguishable from the simpler version $\tilde\eta(\btau_h)$.

In Figure~\ref{fig:lap:n} (right) we plot CPU times for individual parts of the scheme: the assembly and solution of the primal problem \eqref{eq:FEM} (``Primal''), the assembly and solution of the complementary problem~\eqref{eq:dualprob} (``Compl.''), and the evaluation of estimators $\eta$, $\tilde\eta$, $E$ and $\tilde E$ (``Estimates'').
Note that the complementary problem is solved using the block TT version (Algorithm~\ref{alg:als-block}) described in Section~\ref{se:alstau} using the $\mathbf{RT}_1$ space unless stated otherwise.
For comparison we also present the CPU time for the solution of the complementary problem by the simple ALS (Algorithm~\ref{alg:als}) applied to a $(d+1)$-dimensional problem, mentioned at the beginning of Section~\ref{se:alstau} (``Compl. $d+1$''),
as well as by Algorithm~\ref{alg:als-block} but applied to the $\mathbf{RT}_0$ discretization of~\eqref{eq:dualprob}.

We can notice that the computational cost is asymptotically linear in $n$ although we are solving a three-dimensional problem with approximately $n^3$ degrees of freedom.
This is due to uniform boundedness of TT ranks of all TT approximations.
In particular, the TT ranks of $\btau_h$ are bounded by $5$ and TT ranks of $\eta_K(\btau_h)^2$ are bounded by $14$ in this example.
On top of this the high accuracy of the approximate solution is guranteed by the error estimator $\tilde E(\btau_h)$.
Concerning the complementary problem, the simple ALS is much slower than the block TT solver due to a larger number of GMRES iterations needed for solving \eqref{eq:localsys} compared to \eqref{eq:locblock}.

The high computational times for the complementary problem can be reduced by using
the $\mathbf{RT}_0$ space, containing only $n+1$ basis functions in each variable.
However, $\mathbf{RT}_0$ seems to be too coarse, because it yields very pessimistic error estimates, typically $\Ieff^{(\eta)} \approx 59$ and $\Ieff^{(E)} \approx 61$ for this example.
Although the values of effectivity indices tend to stabilize as $n \rightarrow \infty$,
we cannot use this fact in practice to make the bounds more accurate, because there is no guaranteed way of estimating them.
An alternative how to speed up the solution of the complementary problem could be the $\mathbf{RT}_1$ discretization on a coarser grid.
However, this approach is also prone to difficulties, because the interpolation between different grids may amplify the errors, especially those arising in quadrature formulae for entries of matrices and right-hand-side vectors.
This issue may occur, for example, if $u_h$ is irregular on an interface between fine cells, which becomes an interior interface on the coarse grid.
Therefore, we proceed with the block TT version applied to the $\mathbf{RT}_1$ discretization on the original grid in the rest of the paper.

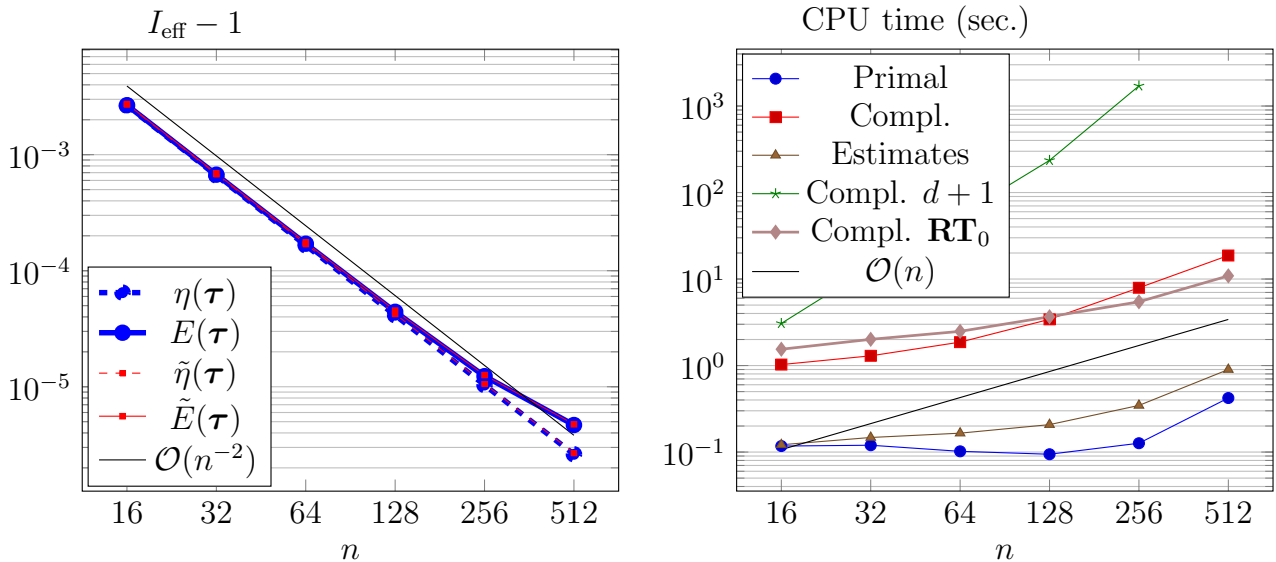
\begin{figure}[htb]
\centering
\caption{Left: differences of effectivity indices from 1 ($\mathbf{RT}_1$ space only); Right: CPU times of different stages of the scheme. The horizontal axes denote the number of grid points in each direction.}
\label{fig:lap:n}
\resizebox{0.49\linewidth}{!}{%
\begin{tikzpicture}
 \begin{axis}[%
  xmode=log,
  ymode=log,
  ylabel={$\Ieff-1$},
  xlabel={$n$},
  xtick={16,32,64,128,256,512},
  xticklabels={16,32,64,128,256,512},
  legend style={at={(0.01,0.01)},anchor=south west},
 ]
  \addplot+[blue,dashed,mark=*,mark options={blue},line width=2pt] coordinates{%
                         (16 , 2.6482e-03)
                         (32 , 6.6188e-04)
                         (64 , 1.6545e-04)
                         (128, 4.1365e-05)
                         (256, 1.0349e-05)
                         (512, 2.6164e-06)
                        };\addlegendentry{$\eta(\btau)$};
  \addplot+[blue,solid,mark=*,mark options={blue},line width=2pt] coordinates{%
                         (16 , 2.6698e-03)
                         (32 , 6.7269e-04)
                         (64 , 1.7090e-04)
                         (128, 4.4265e-05)
                         (256, 1.2340e-05)
                         (512, 4.6815e-06)
                        };\addlegendentry{$E(\btau)$};
  \addplot+[red,dashed,mark=square*,mark options={red},mark size=1] coordinates{%
                         (16 , 2.7128e-03)
                         (32 , 6.7869e-04)
                         (64 , 1.6970e-04)
                         (128, 4.2429e-05)
                         (256, 1.0615e-05)
                         (512, 2.6850e-06)
                        };\addlegendentry{$\tilde \eta(\btau)$};
  \addplot+[red,solid,mark=square*,mark options={red},mark size=1] coordinates{%
                         (16 , 2.7345e-03)
                         (32 , 6.8950e-04)
                         (64 , 1.7515e-04)
                         (128, 4.5330e-05)
                         (256, 1.2606e-05)
                         (512, 4.7501e-06)
                        };\addlegendentry{$\tilde E(\btau)$};
  \addplot+[no marks,black,solid,domain=16:512] {x^(-2)}; \addlegendentry{$\mathcal{O}(n^{-2})$};
 \end{axis}
\end{tikzpicture}
}~
\resizebox{0.49\linewidth}{!}{%
\begin{tikzpicture}
 \begin{axis}[%
  xmode=log,
  ymode=log,
  ylabel={CPU time (sec.)},
  xlabel={$n$},
  xtick={16,32,64,128,256,512},
  xticklabels={16,32,64,128,256,512},
  legend style={at={(0.01,0.99)},anchor=north west},
 ]
  \addplot+ coordinates{%
                         (16 ,1.1733e-01)
                         (32 ,1.1994e-01)
                         (64 ,1.0197e-01)
                         (128,9.4364e-02)
                         (256,1.2652e-01)
                         (512,4.2139e-01)
                        };\addlegendentry{Primal};
  \addplot+ coordinates{%
                         (16 , 2.2270e-02 +  1.0041e+00)
                         (32 , 2.6232e-02 +  1.2692e+00)
                         (64 , 3.2905e-02 +  1.8324e+00)
                         (128, 4.7921e-02 +  3.3751e+00)
                         (256, 7.9418e-02 +  7.8452e+00)
                         (512, 1.8143e-01 +  1.8532e+01)
                        };\addlegendentry{Compl.};
  \addplot+[mark=triangle*] coordinates{%
                         (16 , 3.3175e-02   +    7.8818e-02 + 9.4737e-03)
                         (32 , 3.8562e-02   +    1.0008e-01 + 8.7838e-03)
                         (64 , 5.0437e-02   +    1.0524e-01 + 9.6919e-03)
                         (128, 7.3567e-02   +    1.2390e-01 + 1.0124e-02)
                         (256, 1.3725e-01   +    1.9907e-01 + 9.7289e-03)
                         (512, 5.3848e-01   +    3.5159e-01 + 9.7129e-03)
                        };\addlegendentry{Estimates};

  \addplot+[green!50!black,mark options={green!50!black}] coordinates{%
                         (16 , 3.0471e-02    +   3.0381e+00)
                         (32 , 4.6117e-02    +   1.3552e+01)
                         (64 , 4.3478e-02    +   5.5175e+01)
                         (128, 6.1968e-02    +   2.3533e+02)
                         (256, 8.5373e-02    +   1.7075e+03)
                        };\addlegendentry{Compl. $d+1$};

  \addplot+[pink!70!black,mark=diamond*,mark options={pink!70!black},line width=1pt] coordinates{%
                         (16 , 2.1574e-02  + 1.5222e+00)
                         (32 , 2.6315e-02  + 1.9833e+00)
                         (64 , 3.2594e-02  + 2.4523e+00)
                         (128, 4.0538e-02  + 3.6249e+00)
                         (256, 6.0800e-02  + 5.3961e+00)
                         (512, 1.1117e-01  + 1.0776e+01)
                        };\addlegendentry{Compl. $\mathbf{RT}_0$};

  \addplot+[black,solid,no marks,domain=16:512] {x/150}; \addlegendentry{$\mathcal{O}(n)$};
 \end{axis}
\end{tikzpicture}
}
\end{figure}

Next, we present how the estimators behave with respect to the dimension $d$.
We keep $\kappa=0$ and fix $n=128$.
Due to the conditioning of the complementary problem in higher dimensions, we set
a larger shift $\kappa_0 = 1$ and relax the tolerance of the TT solver for the complementary problem to $\delta_c=10^{-4}$.
With this setting the block TT solver converges, but the quality of the reconstructed flux $\btau_h$ is lower.
The tolerance of the primal problem is kept to be $\delta_p = 10^{-3}$.
Figure~\ref{fig:lap:d} presents differences of indices of effectivity from 1 (left)
and the corresponding CPU times (right).

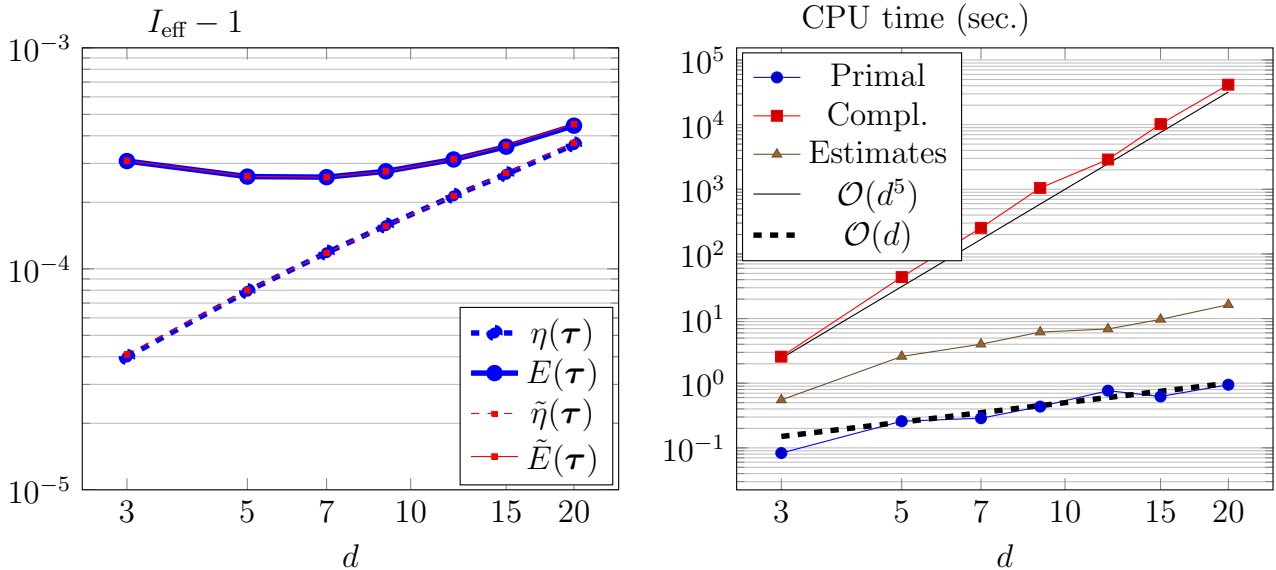
\begin{figure}[htb]
\centering
\caption{Differences of indices of effectivity from 1 (left) and CPU times (right) vs. dimension $d$ for the Poisson problem.}
\label{fig:lap:d}
\resizebox{0.49\linewidth}{!}{%
\begin{tikzpicture}
 \begin{axis}[%
  xmode=log,
  ymode=log,
  xlabel={$d$},
  ymin=1e-5,ymax=1e-3,
  ylabel={$\Ieff-1$},
  xtick={3,5,7,10,15,20},
  xticklabels={3,5,7,10,15,20},
  legend style={at={(0.99,0.01)},anchor=south east},
 ]
  \addplot+[blue,dashed,mark=*,mark options={blue},line width=2pt] coordinates{%
                         (3 ,  3.9897e-05)
                         (5 ,  7.9178e-05)
                         (7 ,  1.1833e-04)
                         (9 ,  1.5685e-04)
                         (12,  2.1375e-04)
                         (15,  2.6778e-04)
                         (20,  3.6498e-04)
                        };\addlegendentry{$\eta(\btau)$};
  \addplot+[blue,solid,mark=*,mark options={blue},line width=2pt] coordinates{%
                         (3 ,  3.0742e-04)
                         (5 ,  2.6147e-04)
                         (7 ,  2.5993e-04)
                         (9 ,  2.7623e-04)
                         (12,  3.1259e-04)
                         (15,  3.5708e-04)
                         (20,  4.4446e-04)
                        };\addlegendentry{$E(\btau)$};
  \addplot+[red,dashed,mark=square*,mark options={red},mark size=1] coordinates{%
                         (3 ,  4.0969e-05)
                         (5 ,  8.0006e-05)
                         (7 ,  1.1831e-04)
                         (9 ,  1.5647e-04)
                         (12,  2.1441e-04)
                         (15,  2.7250e-04)
                         (20,  3.7209e-04)
                        };\addlegendentry{$\tilde \eta(\btau)$};
  \addplot+[red,solid,mark=square*,mark options={red},mark size=1] coordinates{%
                         (3 ,  3.0849e-04)
                         (5 ,  2.6230e-04)
                         (7 ,  2.5991e-04)
                         (9 ,  2.7585e-04)
                         (12,  3.1325e-04)
                         (15,  3.6180e-04)
                         (20,  4.5157e-04)
                        };\addlegendentry{$\tilde E(\btau)$};
 \end{axis}
\end{tikzpicture}
}~
\resizebox{0.49\linewidth}{!}{%
\begin{tikzpicture}
 \begin{axis}[%
  xmode=log,
  ymode=log,
  ylabel={CPU time (sec.)},
  xlabel={$d$},
  xtick={3,5,7,10,15,20},
  xticklabels={3,5,7,10,15,20},
  legend style={at={(0.01,0.99)},anchor=north west},
 ]
  \addplot+ coordinates{%
                         (3 , 8.3113e-02)
                         (5 , 2.5825e-01)
                         (7 , 2.8727e-01)
                         (9 , 4.3489e-01)
                         (12, 7.6674e-01)
                         (15, 6.2360e-01)
                         (20, 9.4663e-01)
                        };\addlegendentry{Primal};
  \addplot+ coordinates{%
                         (3 , 8.1936e-02   +   2.4907e+00)
                         (5 , 2.7786e-01   +   4.3624e+01)
                         (7 , 5.9651e-01   +   2.5252e+02)
                         (9 , 1.1738e+00   +   1.0476e+03)
                         (12, 2.6967e+00   +   2.8882e+03)
                         (15, 2.8135e+00   +   1.0221e+04)
                         (20, 8.0936e+00   +   4.1458e+04)
                        };\addlegendentry{Compl.};
  \addplot+[mark=triangle*] coordinates{%
                         (3 , 3.4400e-01  +  2.0069e-01  +  6.5260e-03)
                         (5 , 1.5138e+00  +  9.3868e-01  +  1.4234e-01)
                         (7 , 2.5222e+00  +  1.4055e+00  +  8.7415e-02)
                         (9 , 3.7097e+00  +  2.4287e+00  +  4.2513e-02)
                         (12, 3.7594e+00  +  3.0799e+00  +  6.7951e-02)
                         (15, 4.2258e+00  +  5.4313e+00  +  5.7231e-02)
                         (20, 4.3391e+00  +  1.1949e+01  +  1.2408e-01)
                        };\addlegendentry{Estimates};

  \addplot+[black,solid,no marks,domain=3:20] {x^5/100}; \addlegendentry{$\mathcal{O}(d^5)$};
  \addplot+[black,dashed,no marks,domain=3:20,line width=2pt] {x/20}; \addlegendentry{$\mathcal{O}(d)$};
 \end{axis}
\end{tikzpicture}
}
\end{figure}

The inefficiency of the error estimators grows slowly with the dimension, but even for the $20$-dimensional problem the estimators
remain accurate up to 3 decimal digits, which is essentially perfect for practical purposes.
Indeed, the error estimator $\tilde E(\btau_h)$ guarantees that the relative error $\trinorm{u - u_h}/\trinorm{u}$ is below $0.00782$, while the exact error is approximately $0.00781$ irrespectively of the dimension $d$.
However, the CPU time of solving the complementary problem grows significantly. The observed order is about $\mathcal{O}(d^5)$.
This is partially because the complementary solution is a vector function of size $d$ and
the TT ranks and the number of GMRES iterations for the reduced problems \eqref{eq:locblock},
both seem to grow linearly in $d$.
Nevertheless, this polynomial growth of the complexity is still much better than the exponential growth, but much worse than the linear cost of solving just the primal problem.

Finally, we test the influence of the reaction coefficient $\kappa^2$ on the accuracy of the error estimator.
We fix $n=128$ and $d=3$ and consider various positive values of $\kappa$.
Therefore the shift $\kappa_0$ is not needed and we formally set $\kappa_0 = 0$.
Here, we use the TT approximation tolerance for the complementary problem $\delta_c=10^{-7}$ in order to solve the problem
for the smallest $\kappa^2=10^{-3}$.
Note that the reaction-diffusion problem \eqref{eq:modpro} is singularly perturbed
for large $\kappa$ and its solution often has steep boundary layers of width $\kappa$.
If these layers are not resolved by the mesh, the finite element solution exhibits spurious oscillations.
The proposed estimator captures this behaviour well and estimates the error accurately even in the singularly perturbed case.
Figure~\ref{fig:lap:kappa} (left) shows that the index of effectivity is almost 1 for small values of $\kappa^2$ and grows for larger values. However, even for $\kappa^2=10^6$ all indices of effectivity are below $1.4$.

As discussed in Section~\ref{se:evalee}, it is more efficient numerically to estimate $E(\btau_h)$ by computing separately $\|\tilde\eta_1\|_2, \|\tilde\eta_2\|_2$, and
the oscillation term $\osc$ computed as the 2-norm of a vector interpolating $\min\{h_K \pi^{-1}, \max_K \kappa^{-1}\}\|r-\Pi r\|_{K}$,
and estimating
\[
\tilde E(\btau_h) = \sqrt{\|\tilde\eta_1\|^2_2 + \|\tilde\eta_2\|^2_2} + \osc + \kappa_0 C_p \|\tilde\eta_2\|_2.
\]
We observe that this (upper) bound of the error is still pretty accurate.
To find out the reasons, we study the values of  $\|\tilde\eta_1\|_2, \|\tilde\eta_2\|_2$ and $\osc$ for two values of $\kappa^2$ in Table~\ref{tab:eta_split}.
\begin{table}[htb]
\centering
\caption{Components of the total error estimator for the 3D elliptic equation, $n=128$.}
\label{tab:eta_split}
\begin{tabular}{c|ccc}
 $\kappa^2$ & $\|\tilde\eta_1\|_2$ &  $\|\tilde\eta_2\|_2$ &  $\osc$       \\\hline
 $10^{-2}$  & 1.6667e-02           & 2.4749e-06            & 8.7726e-07    \\
 $10^2$     & 1.6667e-02           & 1.3728e-04            & 5.0233e-06    \\
\end{tabular}
\end{table}
The first term corresponds to $\|\btau_h - \bnabla u_h\|_{[L^2(\Omega)]^d}$ and it dominates in both regimes of $\kappa^2$. Since it can be computed exactly in the TT format, all error estimates are accurate. As a by-product, the shift term $\kappa_0 C_p \|\tilde\eta_2\|_2$ is also small, since often both $\kappa_0$ and $C_p$ are less than $1$.

\begin{figure}[htb]
\centering
\caption{Differences of indices of effectivity from 1 (left) and CPU times (right) vs. the coefficient $\kappa^2$ for the 3D reaction-diffusion problem.}
\label{fig:lap:kappa}
\resizebox{0.49\linewidth}{!}{%
\begin{tikzpicture}
 \begin{axis}[%
  xmode=log,
  ymode=log,
  xlabel={$\kappa^{2}$},
  ylabel={$\Ieff-1$},
  legend style={at={(0.01,0.99)},anchor=north west},
 ]
  \addplot+[blue,dashed,mark=*,mark options={blue},line width=2pt] coordinates{%
                         (1e-3, 4.1414e-05)
                         (1e-2, 4.1397e-05)
                         (1e-1, 4.1568e-05)
                         (1   , 4.3316e-05)
                         (1e2 , 2.4018e-04)
                         (1e4 , 1.9906e-02)
                         (1e6 , 3.7642e-01)
                        };\addlegendentry{$\eta(\btau)$};
  \addplot+[blue,solid,mark=*,mark options={blue},line width=2pt] coordinates{%
                         (1e-3, 4.1414e-05)
                         (1e-2, 4.1397e-05)
                         (1e-1, 4.1568e-05)
                         (1   , 4.3316e-05)
                         (1e2 , 2.4018e-04)
                         (1e4 , 1.9906e-02)
                         (1e6 , 3.7642e-01)
                        };\addlegendentry{$E(\btau)$};
  \addplot+[red,dashed,mark=square*,mark options={red},mark size=1] coordinates{%
                         (1e-3, 4.2479e-05)
                         (1e-2, 4.2462e-05)
                         (1e-1, 4.2627e-05)
                         (1   , 4.4326e-05)
                         (1e2 , 2.4034e-04)
                         (1e4 , 1.9906e-02)
                         (1e6 , 3.9604e-01)
                        };\addlegendentry{$\tilde \eta(\btau)$};
  \addplot+[red,solid,mark=square*,mark options={red},mark size=1] coordinates{%
                         (1e-3, 4.2479e-05)
                         (1e-2, 4.2462e-05)
                         (1e-1, 4.2627e-05)
                         (1   , 4.4326e-05)
                         (1e2 , 2.4034e-04)
                         (1e4 , 1.9906e-02)
                         (1e6 , 3.9604e-01)
                        };\addlegendentry{$\tilde E(\btau)$};
 \end{axis}
\end{tikzpicture}
}~
\resizebox{0.49\linewidth}{!}{%
\begin{tikzpicture}
 \begin{axis}[%
  xmode=log,
  ymode=log,
  xlabel={$\kappa^{2}$},
  ylabel={CPU time (sec.)},
  legend style={at={(0.99,0.99)},anchor=north east},
 ]
  \addplot+ coordinates{%
                         (1e-3, 8.6689e-02)
                         (1e-2, 9.5746e-02)
                         (1e-1, 9.5698e-02)
                         (1   , 9.4539e-02)
                         (1e2 , 9.2449e-02)
                         (1e4 , 6.7473e-02)
                         (1e6 , 6.7523e-02)
                        };\addlegendentry{Primal};
  \addplot+ coordinates{%
                         (1e-3, 4.8830e-02 +  2.3225e+01)
                         (1e-2, 5.2346e-02 +  3.9949e+00)
                         (1e-1, 4.7584e-02 +  2.9550e+00)
                         (1   , 4.2794e-02 +  3.4194e+00)
                         (1e2 , 4.8069e-02 +  2.1456e+00)
                         (1e4 , 4.9954e-02 +  3.8808e-01)
                         (1e6 , 4.4364e-02 +  1.8501e-01)
                        };\addlegendentry{Compl.};
  \addplot+[mark=triangle*] coordinates{%
                         (1e-3, 7.4410e-02  +  1.7345e-01 +  9.8413e-03)
                         (1e-2, 7.8137e-02  +  1.2536e-01 +  1.1586e-02)
                         (1e-1, 7.5653e-02  +  1.1700e-01 +  1.0236e-02)
                         (1   , 7.4617e-02  +  1.1377e-01 +  1.0332e-02)
                         (1e2 , 7.2683e-02  +  1.0472e-01 +  6.9186e-03)
                         (1e4 , 7.2035e-02  +  1.0801e-01 +  7.5276e-03)
                         (1e6 , 6.6861e-02  +  1.0452e-01 +  7.3886e-01)
                        };\addlegendentry{Estimates};
 \end{axis}
\end{tikzpicture}
}
\end{figure}
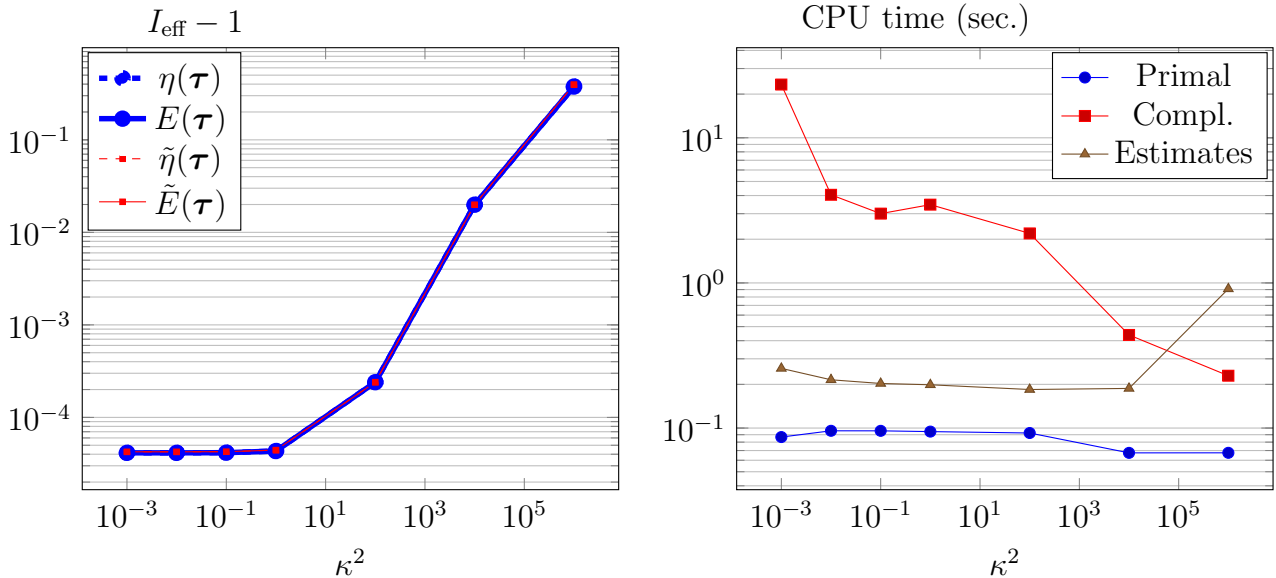



In Figure~\ref{fig:lap:kappa}~(right) we see that the CPU time grows significantly for small $\kappa^2$,
since the complementary problem becomes severely ill-conditioned.
This motivates future research on more efficient preconditioning techniques.
However, even with a simple block Jacobi preconditioner, the CPU time remains stable for a range of orders of $\kappa^2$,
and even decreases for large $\kappa^2$, where the operator approaches a perturbed identity.

\subsection{Poisson equation with variable TT ranks of the solution}
In this example we consider a three-dimensional Poisson equation (i.e. $\kappa^2=0$) in a unit cube, and design the exact solution in the form
\begin{equation}\label{eq:gauss}
 u(x) = \exp\left(-\left(10\left(x_1 \cos^2 \alpha  + x_2 \cos\alpha \sin\alpha + x_3 \sin\alpha\right) - 5\right)^2\right) \prod_{i=1}^{3} x_i(1-x_i),
\end{equation}
parametrized by the angle $\alpha \in (0, \pi/4)$.
The last term accounts for the boundary conditions, and the exponential term defines an anisotropic Gaussian function rotated by the angle $\alpha$ in both $(x_1,x_2)$ and $(x_1,x_3)$ planes.
For $\alpha=0$ the function has a product form, but for $\alpha>0$
it lacks an exact TT decomposition with low ranks, and the ranks of an approximation depend on both the accuracy and the angle $\alpha$.
This also increases the TT ranks of the complementary solution $\btau_h$, and of the error estimators.
We solve the primal problem with the stopping threshold $\delta_p=10^{-3}$ on the grid size $n=64$, the complementary problem with the stopping threshold $\delta_c=10^{-5}$, and with the reaction coefficient shift $\kappa_0=1$.

In Figure~\ref{fig:lap:gauss} (left), we show the TT ranks of the exact solution \eqref{eq:gauss}, truncated to the relative 2-norm error $\delta_c=10^{-5}$ using TT-SVD \cite{osel-tt-2011}, as well as the TT ranks of the complementary solution (for the same threshold).
We see that the ranks may grow logarithmically with the angle.
This results in a sub-algebraic growth of the computing times in Figure~\ref{fig:lap:gauss} (right).
Nevertheless, the scheme remains robust for the whole range of the parameter.
In particular, all effectivity indices are confined within $3\cdot 10^{-3}$ and $5 \cdot 10^{-3}$ from 1 for all considered angles.

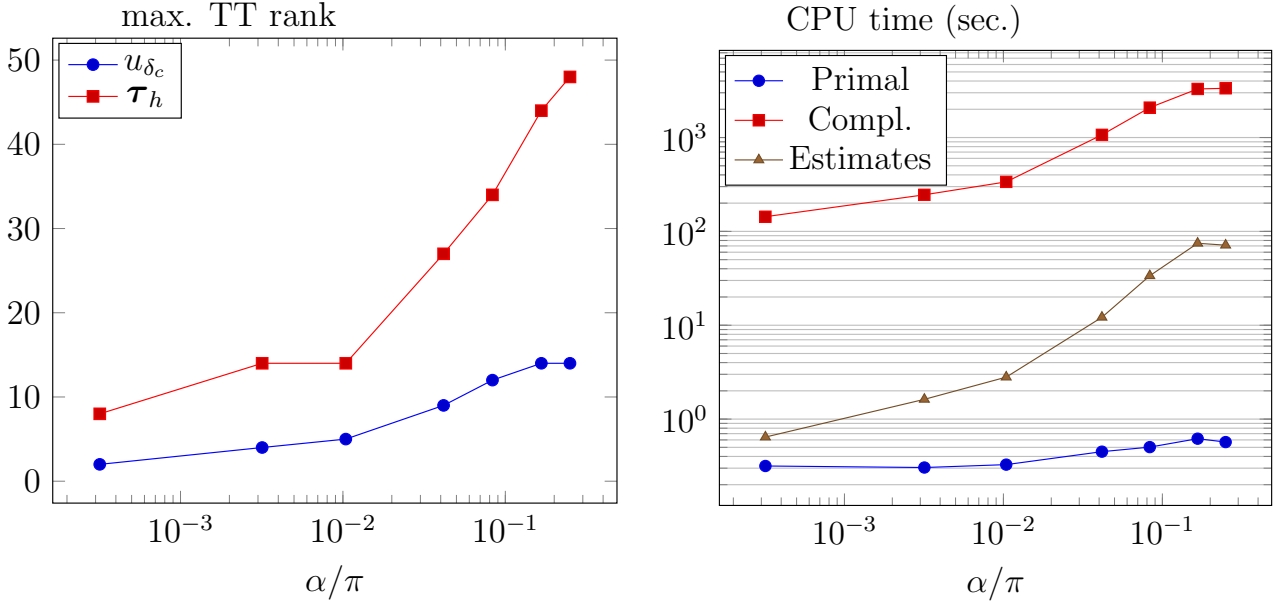
\begin{figure}[htb]
\centering
\caption{3D Poisson equation with variable TT ranks of the solution. Left: maximal TT ranks, right: CPU times. The horizontal axes denote the rotation angle $\alpha$ as defined in~\eqref{eq:gauss}.}
\label{fig:lap:gauss}
\resizebox{0.49\linewidth}{!}{%
\begin{tikzpicture}
 \begin{axis}[%
  xmode=log,
  ymode=normal,
  xlabel={$\alpha/\pi$},
  ylabel={max. TT rank},
  legend style={at={(0.01,0.99)},anchor=north west},
 ]
   \addplot+ coordinates{%
     (1.0/4     , 14)
     (1.0/6     , 14)
     (1.0/12    , 12)
     (1.0/24    , 9 )
     (1.0/96    , 5 )
     (3.1831e-03, 4 )
     (3.1831e-04, 2 )
   }; \addlegendentry{$u_{\delta_c}$};
   \addplot+ coordinates{%
     (1.0/4     , 48)
     (1.0/6     , 44)
     (1.0/12    , 34)
     (1.0/24    , 27)
     (1.0/96    , 14)
     (3.1831e-03, 14)
     (3.1831e-04, 8 )
   }; \addlegendentry{$\btau_h$};
 \end{axis}
\end{tikzpicture}
}~
\resizebox{0.49\linewidth}{!}{%
\begin{tikzpicture}
 \begin{axis}[%
  xmode=log,
  ymode=log,
  xlabel={$\alpha/\pi$},
  ylabel={CPU time (sec.)},
  legend style={at={(0.01,0.99)},anchor=north west},
 ]
  \addplot+ coordinates{%
     (1.0/4     , 5.6895e-01)
     (1.0/6     , 6.1855e-01)
     (1.0/12    , 5.0279e-01)
     (1.0/24    , 4.4978e-01)
     (1.0/96    , 3.2725e-01)
     (3.1831e-03, 3.0484e-01)
     (3.1831e-04, 3.1648e-01)
                        };\addlegendentry{Primal};
  \addplot+ coordinates{%
     (1.0/4     , 1.2840e-01   +    3.3443e+03)
     (1.0/6     , 1.2785e-01   +    3.2942e+03)
     (1.0/12    , 8.9819e-02   +    2.0847e+03)
     (1.0/24    , 7.1850e-02   +    1.0687e+03)
     (1.0/96    , 4.1595e-02   +    3.3690e+02)
     (3.1831e-03, 4.8179e-02   +    2.4526e+02)
     (3.1831e-04, 4.3552e-02   +    1.4333e+02)
                        };\addlegendentry{Compl.};
  \addplot+[mark=triangle*] coordinates{%
     (1.0/4     , 5.3014e+01  +  1.7886e+01  +  3.9773e-01)
     (1.0/6     , 5.9404e+01  +  1.4935e+01  +  3.9115e-01)
     (1.0/12    , 2.1028e+01  +  1.2416e+01  +  2.3206e-01)
     (1.0/24    , 4.2405e+00  +  7.6513e+00  +  2.6660e-01)
     (1.0/96    , 5.0737e-01  +  2.2547e+00  +  3.7941e-02)
     (3.1831e-03, 1.7445e-01  +  1.4315e+00  +  1.7666e-02)
     (3.1831e-04, 1.2995e-01  +  5.0004e-01  +  1.3477e-02)
                        };\addlegendentry{Estimates};
 \end{axis}
\end{tikzpicture}
}
\end{figure}

\subsection{Schr\"odinger equation with Henon--Heiles potential}

In the second example we consider a variable reaction coefficient.
We carry out one step of the Shift-and-Invert iteration for searching the ground state in the molecular Schr\"odinger equation with the model Henon--Heiles potential \cite[Sec. III(A)]{meyer-henon-2002}.
The elliptic PDE \eqref{eq:modpro} is posed on a hypercube $\Omega = (-5,5)^d$. The reaction coefficient is the following polynomial of degree three:
$$
\kappa^2(x) = V_\mathrm{h}(x) + V_\mathrm{u}(x) + 1,
$$
where the harmonic potential $V_\mathrm{h}(x)$ and its non-harmonic perturbation $V_\mathrm{u}(x)$ are given by
$$
  V_\mathrm{h}(x) = \sum_{k=1}^{d} x_k^2
  \quad\text{and}\quad
  V_\mathrm{u}(x) = 0.223606 \sum_{k=1}^{d-1} \left(x_k^2 x_{k+1} - \frac{1}{3} x_{k+1}^3\right).
$$
This corresponds to the Henon--Heiles potential $V_\mathrm{h}(x) + V_\mathrm{u}(x)$ plus a unit shift due to the Shift-and-Invert power method.

The exact solution is chosen to be the Gaussian function
$$
u(x) = \prod_{k=1}^{d} \exp\left(-\frac{x_k^2}{2}\right),
$$
which represents the typical form of the lowest eigenfunction of the Schr\"odinger operator.
The right-hand side $f(x) = \left(d+1+V_\mathrm{u}(x)\right)u(x)$ is computed accordingly.
Despite the cubic nonlinearities, the reaction coefficient $V_\mathrm{h}(x) + V_\mathrm{u}(x) + 1$ is nonnegative (and actually greater than $1$) in $\Omega$, so its definition in the form of $\kappa^2$ is valid.
From the modelling point of view, this domain is also sufficiently large such that the solution is negligibly small (below $10^{-5}$) on the boundary,
which justifies the use of the homogeneous Dirichlet boundary conditions.

The high-dimensional Henon--Heiles Schr\"odinger model was approached with Tensor Train eigenvalue solvers and a spectral discretization \cite{meyer-mctdh-book-2009,khos-dmrg-2010,dkos-eigb-2014},
but no attempts were made to estimate a posteriori errors in a systematic way (instead, a fixed grid was chosen which proved to be overwhelmingly fine in low-dimensional tests).
Here we compute the guaranteed error estimator using the complementary solution
and illustrate the performance of the described approach in high dimensions.

Since the reaction coefficient is strictly positive, we set $\kappa_0=0$. Consequently, $E(\btau)=\eta(\btau)$ and $\tilde E(\btau)=\tilde\eta(\btau)$.
The approximation thresholds are fixed to $\delta_p=\delta_c=10^{-3}$ and grids with either $n=16$ or $n=64$ elements in each dimension are used.
We continue to use $m=4$ quadrature points in each direction of each element.
This gives a relative quadrature error of the order of $10^{-8}$
(estimated by comparing right-hand side elements computed with $m=4$ and $m=8$ on the largest grid cells corresponding to $n=16$),
which is many orders of magnitude smaller than the discretization error.

\begin{figure}[htb]
\centering
\caption{Differences of effectivity indices from 1 (left) and CPU times (right) vs. dimension $d$ for the Schr\"odinger equation. Solid lines: $n=64$, dashed lines: $n=16$.}
\label{fig:hh:d}
\resizebox{0.49\linewidth}{!}{%
\begin{tikzpicture}
 \begin{axis}[%
  xmode=log,
  ymode=log,
  xlabel={$d$},
  ylabel={$\Ieff-1$},
  xtick={3,5,7,10,15,20,30,40},
  xticklabels={3,5,7,10,15,20,30,40},
  legend style={at={(0.01,0.99)},anchor=north west},
 ]
  \addplot+[blue,solid,mark=*,mark options={blue},line width=2pt] coordinates{%
                         (3  ,  4.7909e-03)
                         (5  ,  6.3232e-03)
                         (7  ,  1.0680e-02)
                         (9  ,  2.7117e-02)
                         (10 ,  4.7366e-02)
                         (12 ,  1.0461e+00)
                         (15 ,  4.6785e+00)
                         (20 ,  7.1243e+00)
                         (30 ,  6.4999e+00)
                         (40 ,  5.6559e+00)
                        };\addlegendentry{$E(\btau)$};
  \addplot+[red,solid,mark=square*,mark options={red},mark size=1] coordinates{%
                         (3  ,  4.9712e-03)
                         (5  ,  6.4867e-03)
                         (7  ,  1.3288e-02)
                         (9  ,  8.8293e-02)
                         (10 ,  2.9205e-01)
                         (12 ,  1.2124e+00)
                         (15 ,  4.9371e+00)
                         (20 ,  8.2526e+00)
                         (30 ,  9.1873e+00)
                         (40 ,  8.6029e+00)
                        };\addlegendentry{$\tilde E(\btau)$};

  \addplot+[blue,dashed,mark=*,mark options={blue},line width=2pt] coordinates{%
                         (3  ,  7.8021e-02)
                         (5  ,  1.0177e-01)
                         (7  ,  1.2394e-01)
                         (9  ,  1.4455e-01)
                         (10 ,  1.5575e-01)
                         (12 ,  1.8740e-01)
                         (15 ,  5.6462e-01)
                         (20 ,  8.5182e-01)
                         (30 ,  1.4031e+00)
                         (40 ,  9.1149e-01)
                        };
  \addplot+[red,dashed,mark=square*,mark options={red},mark size=1] coordinates{%
                         (3  ,  7.9834e-02)
                         (5  ,  1.0277e-01)
                         (7  ,  1.2480e-01)
                         (9  ,  1.4537e-01)
                         (10 ,  1.5798e-01)
                         (12 ,  1.9747e-01)
                         (15 ,  6.5015e-01)
                         (20 ,  1.1021e+00)
                         (30 ,  1.6399e+00)
                         (40 ,  1.2770e+00)
                        };
 \end{axis}
\end{tikzpicture}
}~
\resizebox{0.49\linewidth}{!}{%
\begin{tikzpicture}
 \begin{axis}[%
  xmode=log,
  ymode=log,
  ylabel={CPU time (sec.)},
  xlabel={$d$},
  xtick={3,5,7,10,15,20,30,40},
  xticklabels={3,5,7,10,15,20,30,40},
  legend style={at={(0.01,0.99)},anchor=north west},
 ]
  \addplot+ coordinates{%
                         (3 , 1.8047e-01)
                         (5 , 1.5549e-01)
                         (7 , 2.3539e-01)
                         (9 , 3.1143e-01)
                         (10, 5.2607e-01)
                         (12, 4.7906e-01)
                         (15, 6.8892e-01)
                         (20, 1.1966e+00)
                         (30, 2.4426e+00)
                         (40, 3.1670e+00)
                        };\addlegendentry{Primal};
  \addplot+ coordinates{%
                         (3 , 8.7070e-02  +    9.7595e-01)
                         (5 , 5.4531e-01  +    2.5901e+01)
                         (7 , 1.2430e+00  +    4.0121e+02)
                         (9 , 2.2866e+00  +    2.3221e+03)
                         (10, 2.9685e+00  +    2.8194e+03)
                         (12, 4.4555e+00  +    6.2472e+03)
                         (15, 7.6508e+00  +    1.0468e+04)
                         (20, 2.1844e+01  +    2.7995e+04)
                         (30, 6.0559e+01  +    1.4477e+05)
                         (40, 1.4932e+02  +    4.7496e+05)
                        };\addlegendentry{Compl.};
  \addplot+[mark=triangle*] coordinates{%
                         (3 , 7.5046e-02  +  1.1453e-01  +  1.5428e-02)
                         (5 , 1.9337e-01  +  4.3248e-01  +  3.7412e-02)
                         (7 , 2.4911e-01  +  1.8469e+00  +  6.5398e-02)
                         (9 , 3.0780e-01  +  6.1947e+00  +  5.9316e-02)
                         (10, 4.6076e-01  +  6.5707e+00  +  1.3448e-01)
                         (12, 4.5589e-01  +  1.1313e+01  +  1.2139e-01)
                         (15, 6.6060e-01  +  1.9609e+01  +  1.4311e-01)
                         (20, 8.7763e-01  +  5.3270e+01  +  1.2844e-01)
                         (30, 1.1213e+00  +  3.1391e+02  +  2.5272e-01)
                         (40, 1.3590e+00  +  8.0369e+02  +  3.6856e-01)
                        };\addlegendentry{Est.};

  \addplot+[black,solid,no marks,domain=3:38] {x^4/3}; \addlegendentry{$\mathcal{O}(d^4)$};
  \addplot+[red,dashed,mark=square*,mark options={red}] coordinates{%
                         (3 , 3.8094e-02    +  3.2771e-01)
                         (5 , 2.4253e-01    +  6.7586e+00)
                         (7 , 4.8550e-01    +  7.5178e+01)
                         (9 , 7.0426e-01    +  5.9078e+02)
                         (10, 1.0725e+00    +  1.1507e+03)
                         (12, 1.4246e+00    +  2.2609e+03)
                         (15, 2.0907e+00    +  3.5639e+03)
                         (20, 3.5403e+00    +  1.1407e+04)
                         (30, 1.8594e+01    +  8.3720e+04)
                         (40, 2.2517e+01    +  3.3919e+05)
                        };
 \end{axis}
\end{tikzpicture}
}
\end{figure}
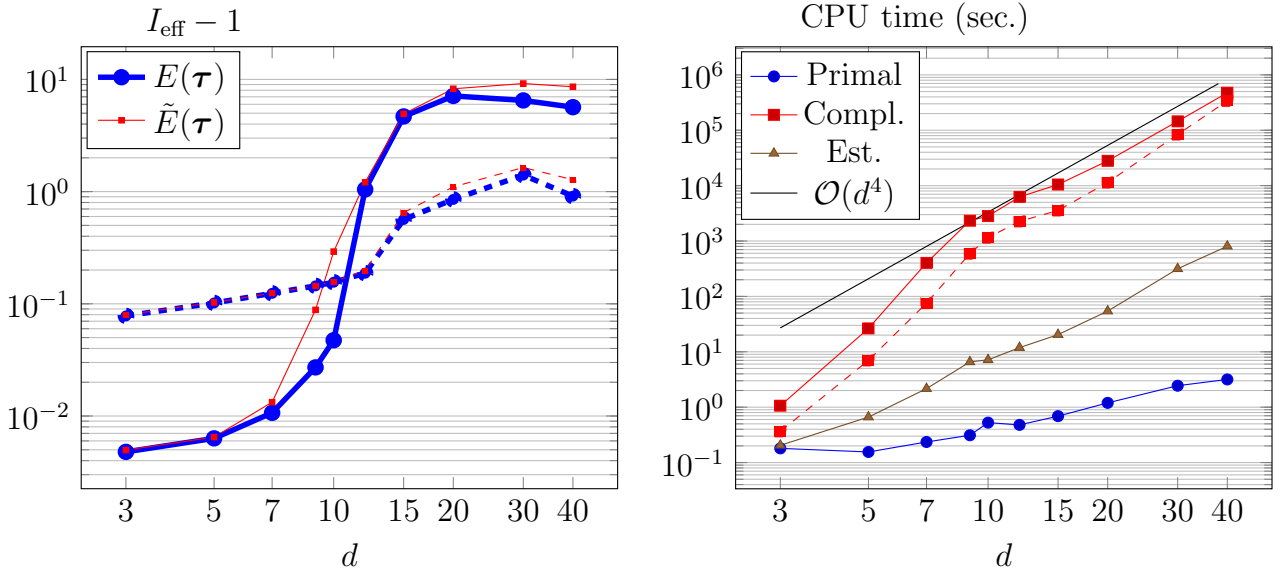

Figure~\ref{fig:hh:d} shows differences of indices of effectivity from 1 (left) and the corresponding CPU times (right).
The efficiency of error estimators deteriorates with the dimension.
However, indices of effectivity tend to stabilize below $10$, which still provides a reasonable error bound.
This deterioration is mainly due to the condition number of the complementary system.
For the smaller $n$ the matrix is well conditioned and the effectivity indices are much closer to 1.
Interestingly, the cost of solving the complementary problem grows with a slower rate $\mathcal{O}(d^4)$ than that for the Poisson problem.
The reason for this is roughly the same number of GMRES iterations for the reduced system \eqref{eq:locblock} for all dimensions,
since the potential term improves the conditioning of this system.

\section{Conclusions}
\label{se:concl}
This paper presents fully computable a posteriori error bounds for low-rank tensor approximate solutions of high dimensional reaction-diffusion problem. These bounds are guaranteed in the sense that they are proved to be above the energy norm of the total error including the discretization error, iteration error in the linear solver, and tensor truncation errors.
The local (elementwise) error estimators may be affected by tensor truncation and quadrature errors.
However, once the complementary solution is obtained, the error estimators can be computed without solving a PDE, and these truncation and quadrature errors can be made negligible
or even zero in certain cases.
This provides a tool for reliable control of both total and local errors in high dimensional computations.

The guaranteed error bound is based on an approximate solution of the complementary problem. Numerical experiments show that the CPU time to solve it grows considerably faster with the dimension than the CPU time needed for the original reaction-diffusion problem. However, the growth is polynomial, which is still feasible in comparison with the exponential complexity of classical approaches.

Interestingly, the error bound is guaranteed for \emph{any} conforming approximation of the primal and complementary solution. Consequently, these solutions can be polluted by arbitrary errors and the estimator still provides a guaranteed error bound.

For efficient computational algorithms, estimates or at least indicators of sizes of individual components of the total error are desirable, see, e.g. \cite{ErnVoh2013,PapStrVoh2018}. Separate estimates of the discretization, algebraic, and tensor truncation errors would enable us to adaptively identify the major source of the error and automatically decide whether to refine the mesh, do more iterations of the algebraic solver, or increase ranks of tensor approximations \cite{Schneider-HT-SFEM-2016}. These separate estimates may be a subject of further research.

Another line of further research should focus on a more efficient way of how to compute the complementary solution. For low dimensional problems, fast and trivially parallelizable local flux reconstructions are known. However, a generalization of these ideas to high dimensions is still open.

\bibliographystyle{plain}
\bibliography{vejchod_aee,vejchod,our,tensor}

\end{document}